\documentclass[10pt,a4paper]{amsart}

\usepackage[left=3.3cm,right=3.3cm,top=4cm,bottom=3.5cm]{geometry}
\usepackage[bookmarksopenlevel=2,colorlinks,linkcolor=darkblue,citecolor=darkblue]{hyperref}

\usepackage{doi}

\usepackage{latexsym}
\usepackage[all,pdf]{xy}
\allowdisplaybreaks[4]
\usepackage{enumerate,enumitem}
\usepackage{indentfirst,setspace}%\linespread{1} % Set linespread
\makeatletter
\newcommand{\leqnomode}{\tagsleft@true}
\newcommand{\reqnomode}{\tagsleft@false}
\makeatother
\usepackage[dvipsnames]{xcolor}
\definecolor{darkblue}{RGB}{15,15,125} % Set hyperlink colour
\usepackage{chngcntr}
\counterwithin{figure}{section}

\usepackage{amsmath,amsthm,amssymb}
\usepackage{mathrsfs,extarrows,mathtools,esint,stmaryrd,bbm}

\newtheorem{thm}{Theorem}[section]
\newtheorem{lem}[thm]{Lemma}
\newtheorem{prop}[thm]{Proposition}

\numberwithin{equation}{section}
\theoremstyle{remark} 
\theoremstyle{definition} \newtheorem{rmk}[thm]{Remark} \newtheorem{df}[thm]{Definition}

\newcommand{\A}{\mathbb{A}}

\newcommand{\R}{\mathbb{R}}
\newcommand{\C}{\mathbb{C}}
\newcommand{\N}{\mathbb{N}}
\newcommand{\B}{\mathbb{B}}

\newcommand{\E}{\mathscr{E}}
\newcommand{\EE}{\mathcal{E}}
\newcommand{\EB}{\mathbf{E}}
\newcommand{\F}{\mathscr{F}}
\newcommand{\Le}{\mathscr{L}}
\newcommand{\HH}{\mathbb{H}}

\newcommand{\RNn}{\mathbb{R}^{N\times n}}

\newcommand{\abs}[1]{\left\lvert #1 \right\rvert}
\newcommand{\abss}[1]{\lvert #1 \rvert}
\newcommand{\norm}[1]{\left\lVert #1 \right\rVert}
\newcommand{\id}{\,\mathrm{d}}
\newcommand{\ssubset}{\subset\joinrel\subset}

\newcommand{\pprime}{\prime\prime}

\newcommand{\mres}{\mathbin{\vrule height 1.7ex depth 0pt width
0.1ex\vrule height 0.1ex depth 0pt width 1ex \,}}

\newcommand{\brac}[1]{\left(#1\right)}

\DeclareMathOperator{\tr}{tr}

 \counterwithout{Fhypoth}{enumi}
 \counterwithout{ohypoth}{enumi}
\newcounter{Fhypoths} \counterwithout{Fhypoths}{enumi}

\makeatletter
\@namedef{subjclassname@2020}{\textup{2020} Mathematics Subject Classification}
\makeatother

\begin{document}

%%fakesection Title page

\title[Partial regularity for $\omega$-minimizers]{Partial regularity for $\omega$-minimizers of quasiconvex functionals}

\author[Z.~Li]{Zhuolin Li}
\thanks{\textit{Funding}: The author is supported by the EPSRC [EP/L015811/1]}
\address{Mathematical Institute, University of Oxford, Andrew Wiles Building \\
  Radcliffe Observatory Quarter, Woodstock Road, Oxford OX2 6GG\\
  United Kingdom}
\email{zhuolin.li@maths.ox.ac.uk}
\date{\today}

\begin{abstract}
We establish partial regularity for the $\omega$-minimizers of quasiconvex functionals of power growth. A first-order partial regularity result of $BV$ $\omega$-minimizers is obtained in the linear growth case under a Dini-type condition on $\omega$. Only assuming the smallness of $\omega$ near the origin, we show partial H\"{o}lder continuity in the subquadratic case by considering a normalised excess.
\end{abstract}

\maketitle

\tableofcontents

\section{Introduction}
	\par We investigate the local regularity of maps $u\colon \Omega \to \R^N$ that almost minimize a variational functional $\F$, which is given by
		\begin{equation}\label{eq:functional}
		\mathscr{F}(u,\Omega) := \int_{\Omega}F(\nabla u) \id x,
		\end{equation} on $W^{1,p}(\Omega,\R^N)$, where the integrand $F\colon \R^{N\times n}\to \R$ is assumed to be strongly quasiconvex (in \textsc{Morrey}'s sense \cite{Morrey52}) and of $p$-growth. See Section \ref{sec:pre} for any undefined notation. \medskip
	\par When the integrand $F$ is of $p$-growth for $p\in [1,\infty)$, the functional $\mathscr{F}$ is obviously well-defined for $u\in W^{1,p}(\Omega,\R^N)$. Now assume that $\Omega$ is a bounded Lipschitz domain. In the case $p>1$, one can apply the direct method to obtain the existence of a minimizer in the Dirichlet class $W^{1,p}_g(\Omega,\R^N)$ for some boundary datum $g\in W^{1,p}(\Omega,\R^N)$. Considering the compactness issue for $p=1$, we study a suitably relaxed problem in $BV$ instead of working with $W^{1,1}$ maps. To extend the integral to maps of bounded variation, we follow \textsc{Lebesgue} \cite{Leb02}, \textsc{Serrin} \cite{Serrin61} and \textsc{Marcellini} \cite{Marcellini86}, and define
		\begin{equation}
		\mathscr{F}_g(u,\Omega):= \inf\left\{\liminf_{j\to \infty} \int_{\Omega} F(\nabla u_j)\id x: \{u_j\}\subset W^{1,1}_g(\Omega,\R^N), u_j\to u \mbox{ in }L^1(\Omega,\R^N) \right\}.
		\end{equation} An integral expression of $\mathscr{F}_g(u,\Omega)$ was found in \cite{KriRin10a} based on the work by \textsc{Ambrosio} \& \textsc{Dal Maso} \cite{AmDal92} and \textsc{Fonseca} \& \textsc{M\"{u}ller} \cite{FonMul93}. When $F$ is quasiconvex, of linear growth and $L^1$-mean coercive, we have
		\begin{equation}\label{eq:relaxex}
		\F_g(u,\Omega) = \int_{\Omega}F(\nabla u)\id x + \int_{\Omega} F^{\infty}\left(\frac{\id D^su}{\id |D^su|}\right)\id |D^su| + \int_{\partial \Omega}F^{\infty}((g-u)\otimes \nu_{\Omega})\id \mathcal{H}^{n-1},
		\end{equation} where $\nu_{\Omega}$ is the outward unit normal on $\partial \Omega$. The third term is present as the trace operator is not continuous in the weak$^{\ast}$ sense in $BV$. We abbreviate the first two terms by 
		\begin{equation}\label{eq:functionalrl}
		\bar{\F}(u,\Omega):=\int_{\Omega}F(Du):=\int_{\Omega}F(\nabla u)\id x + \int_{\Omega} F^{\infty}\left(\frac{\id D^su}{\id |D^su|}\right)\id |D^su|. 
		\end{equation} This expression coincides with the extension by area-strict continuity of (\ref{eq:functional}) from $W^{1,1}$ to $BV$ (see \cite{KriRin10a}, Theorem 4). \medskip
	\par Our focus in this work is on $\omega$-minimizers, which are also called almost minimizers. This concept is closely connected to the elliptic parametric variational problems studied in geometric measure theory (see \cite{Almgren76,Bombieri82,DuSt02}), where the analogues are called $(F,\varepsilon,\delta)$-sets or almost-minimal currents. See \cite{Evans86} for more comments on the connection between the variational problems in our setting and geometric measure theory. It was \textsc{Anzellotti} \cite{Anz83} that first studied $\omega$-minimizers in non-parametric problems, and some later work can be found in \cite{DGG00, DK02, DGK05, Schmidt14}. The solutions to multiple problems (for instance, minimizers subject to some constraints) are $\omega$-minimizers of some suitable functionals. The introduction of this notion, therefore, allows us to unify the study of those problems. We refer to \cite{Anz83, Giusti03, DGG00} for more background information and some examples. 
	\begin{df}
	Suppose that $F\colon\, \RNn \to \R$ is of $p$-growth, and $\F$ and $\bar{\F}$ are defined as in (\ref{eq:functional}) and (\ref{eq:functionalrl}), respectively. 
		\begin{enumerate}[label=(\alph*)]
		\item When $p>1$, a map $u\in W^{1,p}_{loc}(\Omega,\R^N)$ is said to be an {\it $\omega$-minimizer} or {\it almost minimizer} of $\F$ with constant $R_0>0$, if for any ball $B_R=B_R(x_0) \ssubset \Omega$ with $R<R_0$ and any $v \in W^{1,p}_u(B_R,\R^N)$, we have
		\begin{equation}
		\F(u,B_R) \leq \F(v,B_R) +\omega(R)\int_{B_R}(1+\lvert \nabla v\rvert^p)\id x. 
		\end{equation}
		\item When $p=1$, a map $u\in BV_{loc}(\Omega,\R^N)$ is said to be an {\it $\omega$-minimizer} or {\it almost minimizer} of $\bar{\F}$ with constant $R_0>0$, if for any ball $B_R=B_R(x_0) \ssubset \Omega$ with $R<R_0$ and any $v \in BV_u(B_R,\R^N)$, we have
		\begin{equation}
		\bar{\F}(u,B_R) \leq \bar{\F}(v,B_R) +\omega(R)\int_{B_R}(1+\lvert Dv\rvert). 
		\end{equation}
		\end{enumerate}
	\end{df} Alternatively, we can replace the $\omega$-related term by $\omega(R)\int_{B_R}(1+\lvert \nabla u\rvert^p+\lvert \nabla u-\nabla v\rvert^p)\id x$ ($\omega(R)\int_{B_R}(1+\lvert D u\rvert+\lvert D u-Dv\rvert)$ for $p=1$). This definition is more general and appears in some examples. See \cite{DGG00}, \textsection 2 for details. We remark that our results (Theorem \ref{thm:1partialreg} and \ref{thm:upartialreg}) also hold true in this case with only slight modification to the proofs.
	\par Here, the function $\omega$ is defined on $[0,\infty)$ and is nonnegative. Typically, it is assumed to be small enough near the origin, which explains the word ``almost" in the definition above. To be more precise, we assume
		\begin{enumerate}[label=($\boldsymbol{\omega}$\textbf{\theenumi}),leftmargin=3\parindent, itemsep=.2em] 
		\item\label{it:omega1}  $\omega\colon\, [0,\infty)\to [0,1]$ is nondecreasing, and $\omega(0) = \lim_{t\to 0}\omega(t)=0$.
		\end{enumerate}	
	\par For our first result, the first-order partial regularity, the following properties are furthermore required:	
		\begin{enumerate}[label=($\boldsymbol{\omega}$\textbf{\theenumi}),leftmargin=3\parindent, itemsep=.2em] \setcounter{enumi}{1}
		\item \label{it:omega2} There exists $\beta \in (0,1)$ such that $t\mapsto \frac{\omega(t)}{t^{2\beta}}$ is non-increasing in $(0,\infty)$;
		\item \label{it:omega3}The Dini-type condition: for any $\rho>0$, $\Xi_{\frac{1}{4}}(\rho)<\infty$, where \[ \Xi_{\alpha}(\rho) :=\int_0^{\rho} \frac{\omega^{\alpha}(t)}{t} \id t. \]
		\end{enumerate} Sometimes, a more specific control of $\omega$ is assumed:
		\begin{enumerate}[label=($\boldsymbol{\omega}$\textbf{4}),leftmargin=3\parindent, itemsep=.2em] 
		\item\label{it:omega2'} $\omega (t)\leq At^{2\beta^{\prime}}$ for some $\beta^{\prime} \in (0,1)$.
		\end{enumerate} In this case, condition \ref{it:omega3} is satisfied, while \ref{it:omega2} might not hold anymore. The condition \ref{it:omega2'} can significantly simplify the discussion about $\omega$. \medskip
	\par To state the first result, we also specify the assumptions on $F$. See (\ref{eq:refint}) for the definition of $E_p$.
		\begin{enumerate}[label=(\textbf{H\theenumi}),leftmargin=3\parindent, itemsep=.2em]
		\item\label{it:intp} $\lvert F(z)\rvert \leq L(1+\lvert z\rvert^p)$ for any $z\in \R^{N\times n}$ with $L>0$; 
		\item\label{it:intqc} $F$ is strongly quasiconvex in the sense that $F-\ell E_p$ is quasiconvex for some $\ell>0$;
		\item\label{it:intreg} $F$ is in $C^{2,1}_{loc}(\R^{N\times n})$.
		\end{enumerate} 
	\par Our first result is the partial regularity for the derivatives of $\omega$-minimizers:
	\begin{thm}\label{thm:1partialreg}
	Suppose that the function $F$ satisfies \textnormal{\ref{it:intp}-\ref{it:intreg}} with $p=1$, and $\omega$ satisfies \textnormal{\ref{it:omega1}-\ref{it:omega3}}. If $u\in BV_{loc}(\Omega,\R^N)$ is an $\omega$-minimizer of $\bar{\F}$, then it is partially regular in the following sense: there exists a relatively closed $\Le^n$-null set $S_u \subset \Omega$ such that $u$ is $C^1$ on $\Omega \setminus S_u$. Furthermore, the gradient $D u$ has a local modulus of continuity $\rho \mapsto \rho^{\alpha}+\Xi_{\frac{1}{2}}(\rho)$ on $\Omega \setminus S_u$ for any $\alpha \in (0,1)$.
	\end{thm}
	\par In particular, if \ref{it:omega2'} holds, we have $u \in C^{1,\beta^{\prime}}_{loc}(\Omega \setminus S_u)$.
	\par Partial regularity for $\omega$-minimizers under the Dini-type condition (like \ref{it:omega3}) has been done in the super-linear case (see \cite{DGG00,DK02,DGK05}), and the result above gives the counterpart for the end point case ($p=1$).
	\par An excess decay estimate plays an important role in our proof of Theorem \ref{thm:1partialreg}. In particular, we need to estimate the series $\sum_{j=0}^{\infty} \omega^{\alpha}(\tau^j R)$ for some $\alpha, \tau \in (0,1)$ when iterating this process. Such an estimate is essential to control $(Du)_{x_0,R}$ as $R\to 0$, and is guaranteed by \ref{it:omega2} and \ref{it:omega3}. Then it is natural to ask what happens if we only assume the smallness of $\omega$ near the origin \ref{it:omega1}. In this case, the regularity of $Du$ as above is no longer expected, but it is still possible to get the partial H\"{o}lder continuity of $u$ in the subquadratic case (cf. \cite{DGK05}). See Subsection \ref{subsec:1iteration} for details. For the second result, a more precise characterisation of the second derivatives of $F$ is required and we replace \ref{it:intreg} by the following with $L>0$:
	\begin{enumerate}[label=(\textbf{H3$_\theenumi$}),leftmargin=3\parindent, itemsep=.2em]
		\item\label{it:int2deriv} $F$ is $C^2$ with $\lvert F^{\pprime}(z)\rvert \leq L(1+\abs{z})^{p-2}$ for any $z\in \RNn$; 
		\item\label{it:int2deLip} $F^{\pprime}$ is Lipschitz and satisfies \[ \lvert F^{\pprime}(z_1)-F^{\pprime}(z_2)\rvert \leq \frac{L\lvert z_1-z_2\rvert}{(1+\abs{z_1}+\abs{z_2})^{3-p}},  \quad \mbox{for any }z_1,z_2 \in \R^{N\times n}. \]
	\end{enumerate} 
	\begin{thm}\label{thm:upartialreg}
	Suppose that the function $F$ satisfies \textnormal{\ref{it:intp}, \ref{it:intqc}, \ref{it:int2deriv}} and \textnormal{\ref{it:int2deLip}} with $1<p<2$, and $\omega$ satisfies \textnormal{\ref{it:omega1}}. If $u\in W^{1,p}_{loc}(\Omega,\R^N)$ is an $\omega$-minimizer of $\F$, then it is partially regular in the following sense: there exists a relatively closed $\Le^n$-null set $S^{\prime}_u\subset \Omega$ such that $u \in C^{0,\alpha}_{loc}(\Omega\setminus S^{\prime}_u,\R^N)$ for any $\alpha\in(0,1)$.
	\end{thm}
	
	\par The strategy used for most partial regularity results, which is also followed by us, dates back to \textsc{De Giorgi} and \textsc{Almgren}, who worked on minimal surfaces in the context of geometric measure theory. This method was later adapted by \textsc{Giusti} and \textsc{Miranda} \cite{GM68} to prove the partial regularity for minimizers in some variational problems, and by \textsc{Morrey} \cite{Morrey67} for the solutions to certain elliptic systems. It was \textsc{Evans} \cite{Evans86} that showed the first partial regularity result in the quasiconvex setting. Shortly afterwards, \textsc{Fusco} and \textsc{Hutchinson} \cite{FuHu85}, and \textsc{Giaquinta} and \textsc{Modica} \cite{GM86} extended the result to functionals with general integrands $F(x,u,\nabla u)$, and \textsc{Acerbi} and \textsc{Fusco} \cite{AF87} dealt with integrands of $p$-growth with $p\geq 2$. \textsc{Carozza, Fusco} and \textsc{Mingione} \cite{CFM98} first studied the subquadratic case ($1<p<2$), and there are various results afterwards, including \cite{AF89, CP96, DLSV, DM04}. As to the linear growth case ($p=1$), there are only limited references.  \textsc{Anzellotti} and \textsc{Giaquinta} \cite{AG88} showed a partial regularity result in the convex case, and some later references for convex functionals include \cite{Schmidt14, BS13, Bild03, BF02, GMS79}. Some recent progress in the quasiconvex case is given by \textsc{Gmeineder} and \textsc{Kristensen} \cite{KriGm19}.  
	\par The literature on regularity in quasiconvex settings is extensive, and the list above is far from complete. We refer to \cite{KriGm19} and the monograph by \textsc{Giusti} \cite{Giusti03} for a thorough review. The question about the size of singular sets in partial regularity results remains open, but see \cite{KriMin05, KriMin06, KriMin07} for some estimates of the Hausdorff dimensions of singular sets in different set-ups.\medskip
	\par The key step in our proof is to establish the aforementioned excess decay estimate, which is similar to the one for linear homogeneous elliptic systems with constant coefficients (see, for example, \cite{Gia83}, \textsection III.2). With a harmonic approximation process and a Caccioppoli-type inequality, one can transfer the estimate for solutions to elliptic systems to ($\omega$-)minimizers.
	
	\par The proofs of the two theorems (Theorem \ref{thm:1partialreg} and \ref{thm:upartialreg}) are in the same spirit and there are several difficulties especially in our situation. One difficulty appears in the harmonic approximation, where it is impossible to work in the natural space $W^{1,2}$ for a linear elliptic system. This is due to the lack of integrability in the case $1\leq p<2$. We also emphasize that in the linear-growth case, a weak reverse H\"{o}lder inequality is unavailable. Thus, one cannot apply Gehring's lemma to obtain a higher integrability, which is usually helpful in showing the excess decay estimate. The approximation process in Subsection \ref{subsec:harapprox} is adapted from an approach by \textsc{Gmeineder} and \textsc{Kristensen} (\cite{KriGm19}, \textsection 4.3), and in that process they used a Fubini type property of $BV$ maps and truncation to construct an explicit test map. The difference between minimizers and $\omega$-minimizers also leads to an issue, as for the latter there are no Euler-Lagrange equations holding true. However, thanks to the almost-minimality, we are able to establish an Euler-Lagrange type inequality with the help of Ekeland's variational principle.
	\par Another obstacle turns up in the proof of Theorem \ref{thm:upartialreg}. Since the continuity of $\nabla u$ does not hold anymore, the excess decay estimate cannot be carried out as in Theorem \ref{thm:1partialreg}. Instead of estimating the typical excess, we normalise it by $1+\lvert (\nabla u)_{x_0,R}\rvert$ and then try to control the oscillation of $\nabla u$ on that scale. This method is inspired by \cite{FosMin08}, where the authors studied elliptic systems (variational functionals) with coefficients $a(x,u,Du)$ (integrands $F(x,u,Du)$) only continuous in $(x,u)$ in the case $p\geq 2$. The solutions (minimizers) in this case may be considered as almost minimisers of a family of functionals (see \cite{DGG00}, \textsection 2). However, the subquadratic counterpart does not directly follow from the approach in \cite{FosMin08} due to the inhomogeneity of our excess integrand $E_p$. Thus, to switch among different normalising factors, we need to control the ratios between them. A zero-order regularity result for $\omega$-minimizers is done in \cite{DK02} under similar assumptions for the quadratic case ($p=2$), and there are similar results in the scalar case in \cite{CFP99, Manf88,ManfPhD,Min06}. \medskip
	\par We believe that the approach used to prove Theorem \ref{thm:1partialreg} also applies to $\omega$-minimizers in the super-linear case ($p>1$), which were studied in \cite{DGG00,DK02,DGK05}. In the subquadratic case, an Euler-Lagrange type inequality was also obtained in \cite{DGK05} (Lemma 5), with which the harmonic approximation was carried out indirectly. This method can also be adapted into our case (see Subsection \ref{subsec:1ind} and Remark \ref{rmk:uind}).
	\par It is worth mentioning that some variational problems originated from, for example, plasticity, are posed in the space of maps of bounded deformation, where symmetric gradients $Eu := \frac{1}{2}(Du+Du^{t})$ are considered instead of gradients $Du$. Two recent pieces of work by \textsc{Gmeineder} \cite{Gme20, Gme21} present the Sobolev and partial $C^{1,\alpha}$ regularity theory for $BD$ minimizers. We refer to them and the references therein for the background and existing results in this direction. Moreover, one can consider general elliptic operators and the corresponding variational problems. The trace theorem and the existence of minimizers are established in \cite{BDG20} for functionals defined with $\C$-elliptic operators. \textsc{Franceschini} \cite{Federico} studied the case of $\R$-elliptic operators and proved the corresponding partial regularity result. \medskip
	\par The organisation of this paper is as follows. Section \ref{sec:pre} contains some preliminaries, which include the basics of functionals defined on measures and $BV$ maps. Subsection \ref{subsec:elliptic} presents some background results on elliptic systems, which will be used in the harmonic approximation.step In Section \ref{sec:aux} we state some auxiliary results about and properties of the integrands involved. The proof of Theorem \ref{thm:1partialreg} is given in Section \ref{sec:partialreg}, and is split into six steps. The first goal is to obtain a H\"{o}lder-type continuity result of $Du$, after which we further utilise the boundedness to show regularity to the full extent. At the end we also sketch how to approach our result with an indirect argument. Section \ref{sec:upartialreg} is devoted to Theorem \ref{thm:upartialreg}, and some details are omitted since the main steps are similar with those in Section \ref{sec:partialreg}.
	
\section{Preliminaries}\label{sec:pre}
%	\par This section is for some preliminaries which will be useful in the proof of the main result. We first discuss functionals defined on measures rather than Sobolev maps, which is followed by some knowledge on maps of bounded variation. The third section is about the existence and regularity of the solutions to Legendre-Hadamard elliptic systems.

\subsection{Basic notation}\label{subsec:notation}
	\par This subsection is for clarifying the notation used throughout the paper. 
	\par The $n$-dimensional Euclidean space $\R^n$ is equipped with the Lebesgue measure $\Le^n$. Throughout the paper, the symbol $\Omega$  indicates a bounded open set in $\R^n$ with $n\geq 2$ if not specified. For any measurable set $S\subset \R^n$, if $0<\Le^n(S)<\infty$, the average of $f\in L^1(S,\HH)$ is denoted by \[ f_S := \fint_S f \id x := \frac{1}{\Le^n(S)}\int_S f\id x. \] The space $\HH$ here and in the following is a finite dimensional Hilbert space, and we denote its norm by $\lvert \cdot\rvert$. For a ball $B(x,R) \subset \R^n$, we may use $f_{x,R}$ or $f_R$ to represent $f_{B(x,R)}$. If $\mu$ is an $\HH$-valued Radon measure on $\Omega$ and $S \ssubset \Omega$ is a Borel set, the average of $\mu$ on $S$ is similarly denoted by $\mu_A(:= \mu(A)/\Le^n(A))$.
	\par When considering a locally integrable function or map, we intend the precise representative of it. For any $u\in L^1_{loc}(\Omega,\mathbb{H})$, it has an approximate limit $\tilde{u}(x)$ $\Le^n$-almost everywhere, i.e., for $\Le^n$-almost every $x\in \Omega$ there exists $\tilde{u}(x)\in \mathbb{H}$ such that \[ \lim_{r\to 0}\fint_{B(x,r)}\lvert u(y)-\tilde{u}(x)\rvert\id y=0. \] Then $\tilde{u}$ is defined on $\Omega$ except for an $\Le^n$-null set and is called the {\it precise representative} of $u$. The meaning of $u\vert_{\partial B}$ with $B$ being a ball in $\Omega$ is clear when $u$ has proper regularity, and it is considered as both the trace of $u: B\to \HH$ (when defined) and the pointwise restriction of $\tilde{u}$.
	\par The Sobolev spaces $W^{k,p}(\Omega,\R^N)$ are defined as usual, and see Subsection \ref{subsec:BV} for the space of maps of bounded variation. For $u$ in $W^{1,p}(\Omega,\R^N), p\geq 1$ and $BV(\Omega,\R^N)$, we have the Dirichlet classes 
		\begin{align*}
		&W^{1,p}_u(\Omega,\R^N):= \{v\in W^{1,p}(\Omega,\R^N): u-v \in W^{1,p}_0(\Omega,\R^N) \} \quad \mbox{and}\\
		&BV_u(\Omega,\R^N):= \{v \in BV(\Omega,\R^N): w_{u,v} \in BV(\R^n,\R^N),\lvert Dw_{u,v}\rvert(\partial \Omega)=0 \},
		\end{align*} respectively. The map $w_{u,v}$ above is defined as $u-v$ in $\Omega$ and is extended to $\R^n\setminus \Omega$ by $0$. Notice that we can define $W^{1,1}_u(\Omega,\R^N)$ with $u\in BV(\Omega,\R^N)$ for a Lipschitz domain $\Omega$, as the trace of $u$ exists in $L^1(\partial \Omega,\R^N)$ and can be considered as that of a map in $W^{1,1}(\Omega,\R^N)$ (see \cite{Giovanni17}, Chapter 18).
	\par The space of $N\times n$ matrices with real entries is denoted by $\R^{N\times n}$ and equipped with the inner product $z\cdot w = \mbox{tr}(z^{t}w)$ for any $z,w \in \R^{N\times n}$ and the induced norm $\lvert \cdot\rvert$. Let $\bigodot^2(\R^{N\times n})$ be the space of symmetric and real bilinear forms on $\R^{N\times n}$, that is, the space $\bigodot^2(\R^{N\times n})$ consists of maps $\A\colon \R^{N\times n}\times \R^{N\times n} \to \R$ such that \[\A[z,w]=\A[w,z], \quad \A[az_1+z_2,w]=a\A [z_1,w]+\A[z_2,w]  \]	for any $z,z_1,z_2,w \in \R^{N\times n}$ and $a\in \R$. The operator norm of $\A\in \bigodot^2(\R^{N\times n})$ is $\lvert \A\rvert=\sup\{\A[z,w]:\lvert z\rvert,\lvert w\rvert\leq 1\}$. 
	\par Consider an integrand $F\colon \R^{N\times n}\to \R$. It is said to be of {\it $p$-growth} ($p\geq 1$), if there exists $L>0$ such that 
		\begin{equation}
		\abs{F(z)}\leq L(1+\abs{z}^p), \quad \mbox{for any }z\in \RNn.
		\end{equation} In particular, the function is of linear growth if $p=1$. We say the integrand is
		\begin{itemize}[label=\footnotesize{$\bullet$}]
		\item {\it quasiconvex} if for any $z\in \RNn$ and any $\varphi \in C_c^{\infty}((0,1)^n,\R^N)$ we have
			\begin{equation}
			\int_{(0,1)^n} F(z+\nabla \varphi)\id x \geq F(z);
			\end{equation}
		\item {\it rank-one convex} if $F(z+t\xi)$ is convex in $t\in \R$ for any $z,\xi\in \RNn$ with $\mbox{rank}(\xi)\leq 1$ (i.e., $\xi = a\otimes b$ for some $a\in \R^N, b\in \R^n$).
		\end{itemize} We refer to \cite{Dac08} for a thorough discussion about different convexity notions. In particular, we will use the fact that quasiconvexity implies rank-one convexity (see \cite{Dac08}, Theorem 5.3). When $F$ has sufficient differentiability at a fixed point $z\in \RNn$, we consider $F^{\prime}(z)$ as an $N\times n$ matrix and $F^{\pprime}(z)$ as a symmetric bilinear form in $\bigodot^2(\RNn)$.
	\par The {\it reference integrand} in the following is a function defined on any finite dimensional Hilbert space (the space is not emphasized in the notation): 
		\begin{equation}\label{eq:refint}
		E_p(z) := \langle z\rangle^p-1:=(1+\lvert z\rvert^2)^{\frac{p}{2}}-1.
		\end{equation} In particular, we denote $E_1$ by $E$ for convenience. More generally, for any $\mu\geq 0$ define \[E_p^{\mu}(z) := ((1+\mu)^2+\lvert z\rvert^2)^{\frac{p}{2}}-(1+\mu)^p. \] It is obvious that 
		\begin{equation}\label{eq:Ehomog}
		E_p^{\mu}(z) = (1+\mu)^pE_p\brac{\frac{z}{1+\mu}}.
		\end{equation} Given any $A\in \R^{N\times n}$, set $E_p^{A}:=E_p^{\lvert A\rvert}$.
	\par The constants $c$ and $C$ throughout this paper may vary from one line to another, and the factors they depend on will be specified when necessary.

\subsection{Functionals defined on measures}\label{subsec:measure}
	\par In this subsection, we recall some background results about functionals defined on measures, that is, functionals with measures instead of only maps as arguments.
	\par Let $\mu$ be an $\mathbb{H}$-valued Radon measure on an open set $\Omega \subset \R^n$. Then the total variation $|\mu|$ of it is a real-valued Radon measure on $\Omega$. The measure $\mu$ is said to be a bounded Radon measure if $|\mu|(\Omega)<\infty$. By the Lebesgue-Radon-Nikod\'{y}m decomposition, we can decompose $\mu$ as \[ \mu = \mu^{ac}+\mu^s =\frac{\id \mu}{\id \Le^n}\Le^n + \frac{\id\mu}{\id|\mu^s|}|\mu^s|. \]
	\par Let $f\colon \bar{\Omega}\times \mathbb{H}\to \R$ be a Borel function of linear growth. Its {\it recession function} is defined by 
		\begin{equation}
		f^{\infty}(x,z):= \limsup_{\substack{y\to x,w\to z\\ t\to \infty}} \frac{f(y,tw)}{t}, \quad (x,z)\in \bar{\Omega}\times \mathbb{H}.
		\end{equation} Hence, the recession function $f^{\infty}$ is also Borel and positively $1$-homogeneous in the second argument, and satisfies $|f^{\infty}(x,z)|\leq C|z|$ for some $C>0$. Now we can define the signed Radon measure $f(\cdot ,\mu)$: for any Borel set $A$ compactly contained in $\Omega$, set 
		\begin{equation}\label{eq:flrelaxed}
		 f(\cdot ,\mu)(A):= \int_A f(\cdot ,\mu) := \int_A f(\cdot ,\frac{\id \mu}{\id \Le^n})\id \Le^n + \int_Af^{\infty}(\cdot ,\frac{\id \mu}{\id|\mu^s|})\id |\mu^s|. 
		 \end{equation} For any $z\in \HH$, we write $f(\mu-z)$ as a short-hand of $f(\mu-z\Le^n)$. If $\mu$ is bounded, the definition above can be extended to all Borel subsets of $\Omega$ and $f(\cdot\, ,\mu)$ is a bounded Radon measure on $\Omega$. If $f$ is in addition assumed to be continuous and the limit superior in the definition of $f^{\infty}$ is a limit which exists locally uniformly in $(x,z)$, then we say that $f$ admits a {\it regular recession function}. The collection of continuous functions with regular recession functions is denoted by $\EB_1(\Omega,\HH)$ ($\EB_1(\HH)$ for maps $f:\HH \to \R$). It is clear that functions in $\EB_1(\Omega,\HH)$ are of linear growth.
	\par In our case, the functions taken into consideration do not explicitly depend on $x$. For a Borel function $f\colon \mathbb{H}\to \R$ of linear growth, the measure $f(\mu)$ can be defined as above. We now recall the convergence of Radon measures with respect to some particular function $f$.
	\begin{df}\label{df:msrcnvrg}
	Suppose that $\{\mu_j\}$ and $\mu$ are Radon measures defined on $\Omega$ such that $\mu_j \xrightharpoonup{\ast}\mu$ in $\mathcal{M}(\Omega,\mathbb{H})$ and $f(\mu_j)(\Omega)\to f(\mu)(\Omega)$.
		\begin{enumerate}[label=(\alph*)]
		\item $\mu_j$ is said to {\it converge to $\mu$ strictly} if $f =|\cdot|$;
		\item $\mu_j$ is said to {\it converge to $\mu$ area-strictly} if $f =E$.
		\end{enumerate}
	\end{df}
	\begin{lem}\label{lem:areaapprox}
	Any Radon measure $\mu$ on $\Omega$ can be locally area-strictly approximated by smooth maps. If $\mu$ is bounded on $\Omega$, the approximation can be global.
	\end{lem}
	\par This can be done by mollification with the help of Theorem 2.2 and 2.34 in \cite{AFP00}.
	\par A generalisation of a result by Reshetnyak (see \cite{Resh68} and the appendix of \cite{KriRin10a}) states that if $f\colon \Omega \times\HH\to \R$ is in $\EB_1(\Omega,\HH)$, then \[ \int_{\Omega}f(\cdot,\mu_j) \to \int_{\Omega}f(\cdot,\mu) \] for $\mu_j \to \mu$ in the area-strict sense. An immediate corollary is that convergence in the area-strict sense implies that in the strict sense.
	
\subsection{Maps of bounded variation}\label{subsec:BV}
	\par For maps of bounded variation and the relevant results, we refer to \cite{AFP00}. Some definitions and results are stated here for later use.
	\par Consider a bounded open set $\Omega \subset \R^n$. A map $u \colon \Omega \to \R^N$ is said to be of bounded variation if it is in $L^1(\Omega,\R^N)$ and its distributional derivative can be represented by a bounded $\R^{N\times n}$-valued Radon measure, i.e., \[\lvert Du\rvert(\Omega):= \sup\left\{\int_{\Omega}u\cdot \mbox{div}(\varphi) \id x: \varphi\in C^1_c(\Omega,\R^{N\times n}), \lvert \varphi\rvert\leq 1\right\}<\infty. \] The space of maps of bounded variation is a Banach space under the norm $\| u\|_{BV(\Omega)}:= \|u\|_{L^1(\Omega)}+\lvert Du\rvert(\Omega)$.
	\par Convergence with respect to the $BV$ norm is rather strong and rarely used. Instead, we consider two other forms of convergence: Suppose that $\{u_j\} \subset BV(\Omega,\R^N)$, $u\in BV(\Omega,\R^N)$ and $u_j \to u$ in $L^1(\Omega,\R^n)$. We say that $\{u_j\}$ converges to $u$ in the {\it BV (area-)strict sense} if $\{Du_j\}$ converges to $Du$ in the (area-)strict sense as in Definition \ref{df:msrcnvrg}.
	\par It is well-known that smooth maps are dense in $BV(\Omega,\R^N)$ in the $BV$ area-strict sense:
	\begin{lem}\label{lem:BVdense}
	Let $\Omega \subset \R^n$ be a bounded open set without any additional regularity assumptions on $\partial \Omega$. If $u \in BV(\Omega,\R^N)$, there exists a sequence $\{u_j\} \subset W^{1,1}_u\cap C^{\infty}(\Omega,\R^N)$ such that $u_j \to u$ in the $BV$ area-strict sense. If $u \in W^{1,1}(\Omega,\R^N)$, we can further require strong convergence in $W^{1,1}(\Omega,\R^N)$.
	\end{lem} See \cite{KriRin10b}, Lemma 1 for a proof. The following lemma allows us to approximate a map of bounded variation in energy and is helpful in various cases. See Theorem 4 in \cite{KriRin10a} for the details.
	\begin{lem}\label{lem:intapprox}
	Suppose that $G\colon \R^{N\times n}\to \R$ is rank-one convex and of linear growth. If $\Omega \subset \R^n$ is a bounded Lipschitz domain, $u_j,u\in BV(\Omega,\R^N)$ and $u_j \to u$ in the $BV$ area-strict sense, then 
		\begin{equation}\label{eq:intapprox}
		\int_{\Omega}G(Du_j) \to \int_{\Omega}G(Du) \quad \mbox{as }j\to \infty.
		\end{equation}
	\end{lem}
	\par The two lemmas above give a direct corollary:
	\begin{lem}\label{lem:intapprox1}
	Suppose that $\Omega \subset\R^n$ is a bounded Lipschitz domain. For any $u\in BV(\Omega,\R^N)$, there exists a sequence $\{u_j\}\subset W^{1,1}_u\cap C^{\infty}(\Omega,\R^N)$ such that $u_j\to u$ in the $BV$ area strict sense. Furthermore, for any function $G\colon \R^{N\times n}\to \R$ that is rank-one convex and of linear growth, we have
		\begin{equation}
		\int_{\Omega}G(Du_j) \to \int_{\Omega}G(Du) \quad \mbox{as }j\to \infty.
		\end{equation}
	\end{lem}
	\begin{rmk}\label{rmk:intapprox}
	Notice that $E(\cdot -z_0)$ for any $z_0\in \R^{N\times n}$ is convex by (\ref{eq:Ediffest}), and thus rank-one convex. Obviously it is of linear growth and then Lemma \ref{lem:intapprox1} applies to $E(\cdot -z_0)$. The lemma also holds for functions satisfying \ref{it:intqc} with $p=1$ as quasiconvexity implies rank-one convexity.
	\end{rmk}
	\par The next result is a Fubini-type property for $BV$ maps. It involves $BV$ maps on submanifolds of $\R^n$, which are well-defined by local charts and partitions of unity. In our case, we only consider $(n-1)$-spheres, which can be covered by two local charts that correspond to the stereographic projections from two antipodal points. The two charts are taken to be such that they both correspond to a bounded open subset of $\R^{n-1}$, over which the induced metric is comparable to the natural one on $\R^{n-1}$. Thus, we can apply various results for maps defined on (open subsets of) $\R^{n-1}$. For a $BV$ map $u \colon \partial B \to \R^N$, we denote its {\it tangential approximate gradient} by $\nabla_{\tau}u$, which exists $\mathcal{H}^{n-1}$-almost everywhere on $\partial B$. Its {\it tangential distributional derivative} is denoted by $D_{\tau}u$. Indeed, the former is the absolutely continuous part of the latter with respect to $\mathcal{H}^{n-1}\mres \partial B$, and the two coincide when $u \in W^{1,p}(\partial B,\R^N)$ with $p\geq 1$.
	\begin{lem}\label{lem:BVFubini}
	Let $B_R$ denote a ball $B(x_0,R) \subset \R^n$ and $u$ be a map in $BV(B_R,\R^N)$. Then for $\Le^1$-almost every $\rho\in (0,R)$, the pointwise restriction $u\vert_{\partial B_{\rho}}$ coincides with the traces of $u$ from $B_{\rho}$ and $B\setminus \bar{B}_{\rho}$, and is in $BV(\partial B_{\rho},\R^N)$. For any two radii $r_1, r_2$ with $0<r_1<r_2<R$, we can find $\rho \in (r_1,r_2)$ such that the above holds and the total variation of $u\vert_{\partial B_{\rho}}$ on $\partial B_{\rho}$ is bounded by that of $u$:
		\begin{equation}\label{eq:BVFubini}
		\int_{\partial B_{\rho}} \lvert D_{\tau}(u\vert_{\partial B_{\rho}})\rvert \leq \frac{C(n,N)}{r_2-r_1}\int_{B_{r_2}\setminus \bar{B}_{r_1}}\lvert Du\rvert.
		\end{equation}
	\end{lem} This lemma is Lemma 2.3 in \cite{KriGm19} and allows us to work on those balls over the boundary of which a $BV$ map has nice properties. To see this, we recall the definition of fractional Sobolev spaces. Let $X$ be an embedded $d$-submanifold ($d\leq n$) of $\R^n$ and $s\in(0,1)$, $r\in (1,\infty)$. The space $W^{s,r}(X,\R^N)$ consists of maps $u \colon X\to \R^N$ of which the Gagliardo norm \[ \|u\|_{W^{s,r}(X)} = (\|u\|^r_{L^r(X)}+ [u]^r_{W^{s,r}(X)})^{\frac{1}{r}},\] is finite. The semi-norm is defined by \[[u]^r_{W^{s,r}(X)}:= \int_{X}\int_X \frac{\lvert u(x)-u(y)\rvert^r}{\lvert x-y\rvert^{d+sr}}\id \mathcal{H}^d(x)\id \mathcal{H}^d(y),\]
	\begin{lem}\label{lem:BVbdry}
	Let $B$ be a ball $B(x_0,R) \subset \R^n$ and $v \in BV(\partial B,\R^N)$. Then we have $v\in W^{1-\frac{1}{r},r}(\partial B,\R^N)$ and 
		\begin{equation}
		\left(\fint_{\partial B}\int_{\partial B} \frac{\lvert v(x)-v(y)\rvert^r}{\lvert x-y\rvert^{n+r-2}}\id \mathcal{H}^{n-1}(x)\id \mathcal{H}^{n-1}(y)\right)^{\frac{1}{r}} \leq CR^{\frac{1}{r}}\fint_{\partial B}\lvert D_{\tau}v\rvert,
		\end{equation} where $C=C(n,N,r)>0$. The range of $r$ depends on the dimension: \[  \left\{\begin{aligned}
		&r = \frac{n}{n-1},& &n\geq 3\\
		&r\in (1,2),& &n=2.
		\end{aligned} \right. \]
	\end{lem} This lemma is a corollary of several embedding results. See \cite{BBM04}, Lemma D.1 for $n\geq 3$, and \cite{Tartar07}, Lemma 38.1 and \cite{Triebel83}, \textsection 3.3.1  for $n=2$. There is also a discussion after Lemma 2.4 in \cite{KriGm19}.
%	The case $n\geq 3$ is from \cite{BBM04}, Lemma D.1, which gives the continuous embedding 
%		\begin{equation}
%		BV(\R^d) \subset W^{s,p}(\R^d), \quad \mbox{with }d\geq 2, s\in(0,1) \mbox{ and }\frac{1}{p}=1-\frac{1-s}{d}.
%		\end{equation} Here we have $d=n-1, p=\frac{n}{n-1},s=\frac{1}{n}$. 
%	\par In the case $n=2$ we need to combine the following continuous embeddings, for which see \cite{Tartar07}, Lemma 38.1 and \cite{Triebel83}, \S 3.3.1 Theorem, respectively:
%		\begin{align}
%		&BV(\R)\cap L^{\infty}(\R) \subset B_{\infty}^{\frac{1}{2},2}(\R);\\
%		&B^{\frac{1}{2}}_{\infty}(-1,1)\subset B^{1-\frac{1}{p},p}_p(-1,1)=W^{1-\frac{1}{p},p}(-1,1). 
%		\end{align} Notice that a $BV$ map $v$ on $\partial B$ can be decomposed as two $BV$ functions $v_1,v_2$ which are compactly supported on the two aforementioned local charts respectively. $v_i,i=1,2$, is equivalent to a $BV$ function $\bar{v}_i$ compactly supported on $(-1,1)$, which is in $BV(\R)$. By the first embedding we know $v_i \in B_{\infty}^{\frac{1}{2},2}(\R)$ and thus in $B_{\infty}^{\frac{1}{2},2}(-1,1)\subset W^{1-\frac{1}{p},p}(-1,1)$, which gives the desired result.
		
\subsection{Estimates for elliptic systems}\label{subsec:elliptic}
	\par We need some results on Legendre-Hadamard elliptic systems. A bilinear form $\A \in\bigodot^2(\RNn)$ is said to satisfy the {\it strong Legendre-Hadamard condition} if there exists $\lambda, \Lambda >0$ such that 
		\begin{equation}\label{eq:LHcondition}
		\left\{\begin{aligned}
		&\A[\eta\otimes \xi,\eta\otimes \xi] \geq \lambda \lvert \eta\rvert^2\lvert \xi\rvert^2,&  &\mbox{for any }\eta\in \R^N, \xi \in \R^n,\\
		&\A[z,z]\leq \Lambda \lvert z\rvert^2, & & \mbox{for any }z \in \R^{N\times n}.
		\end{aligned} \right.
		\end{equation}  
	\par We say that $u$ is $\A$-harmonic in some open set $\Omega$ if it satisfies 
		\begin{equation}\label{eq:ellipsys}
		-\mbox{div}(\A\nabla u)=0
		\end{equation}  in the distributional sense in $\Omega$.
	\begin{lem}\label{lem:ellip1}
	Suppose that $\A\in \bigodot^2(\R^{N\times n})$ satisfies \textnormal{(\ref{eq:LHcondition})} with some $\Lambda, \lambda>0$. If $h \in W^{1,1}(B_R,\R^N)$ is $\A$-harmonic in the ball $B_R=B(x_0,R) \subset \R^n$, then $h$ is in $C^{\infty}(B_R,\R^N)$, and for any $z\in \R^{N\times n}$ and some $c_a=c_a(n,N,\frac{\Lambda}{\lambda})>0$ we have
		\begin{equation}\label{eq:ellipest1}
		\sup_{B_{\frac{R}{2}}}\lvert \nabla h-z\rvert+R\sup_{B_{\frac{R}{2}}}\lvert \nabla^2 h\rvert \leq c_a \fint_{B_R}\lvert \nabla h-z\rvert\id x.
		\end{equation}
	\end{lem}
	\par This lemma is classical and obtained with, for example, the results in \cite{Gia83}, \textsection III.2 and \cite{Giusti03}, \textsection 7.2. 
	The next result is also classical, and see Proposition 2.11 in \cite{KriGm19} and the references therein for a proof.
	\begin{lem}\label{lem:ellip2}
	Suppose that $\A\in \bigodot^2(\R^{N\times n})$ satisfies \textnormal{(\ref{eq:LHcondition})} with some $\Lambda, \lambda>0$. Let $r\in (1,\infty)$, $q\in [2,\infty)$ and $B$ be the unit ball in $\R^n$.
		\begin{enumerate}[label=\textup{(\alph*)}]
		\item\label{it:ellipDiri} For any $g\in W^{1-\frac{1}{r},r}(\partial B,\R^N)$, the elliptic system 
			\begin{equation}\label{eq:elliptic1}
			\left\{\begin{aligned}
			&-\mathrm{div}(\A\nabla h)=0,& &\mathrm{in }\ B\\
			&h\vert_{\partial B}=g,& &\mathrm{on }\ \partial B
			\end{aligned} \right.
			\end{equation} admits a unique solution $h\in W^{1,r}(B,\R^N)$, and there exists $C=C(n,N,r, \frac{\Lambda}{\lambda})>0$ such that \[ \|h\|_{W^{1,r}(B,\,\R^N)} \leq C \norm{g}_{W^{1-\frac{1}{r},r}(\partial B,\,\R^N)}. \]
		\item For any $f\in L^q(B,\R^N)$, the elliptic system 
			\begin{equation}\label{eq:elliptic2}
			\left\{\begin{aligned}
			&-\mathrm{div}(\A\nabla w)=f,& &\mathrm{in }\ B\\
			&w\vert_{\partial B}=0,& &\mathrm{on }\ \partial B
			\end{aligned} \right.
			\end{equation} admits a unique solution $w\in W^{2,q}\cap W^{1,q}_0(B,\R^N)$, and there exists $C=C(n,N,q, \frac{\Lambda}{\lambda})>0$ such that \[ \|w\|_{W^{2,q}(B,\,\R^N)} \leq C \|f\|_{L^q(B,\,\R^N)}. \]
		\end{enumerate}
	\end{lem} 
	\begin{rmk}\label{rmk:ellipest}
	If we only consider the gradient $\nabla h$ above, it is enough to control $\norm{\nabla h}_{W^{1,r}(B,\R^N)}$ by $[g]_{W^{1-\frac{1}{r},r}(\partial B,\,\R^N)}$ with considering $g-(g)_{\partial B}$. In particular, if $g\in W^{1,r}(B,\R^N)$, its trace $\tr_{B}g$ exists in $W^{1-\frac{1}{r},r}(B,\R^N)$ (see \cite{Giovanni17}, Section 18.4). The estimate of $\norm{h}_{W^{1,r}(B,\R^N)}$ in \ref{it:ellipDiri} can be then replaced by $\norm{g}_{W^{1,r}(B,\R^N)}$.
	\end{rmk}
	
%	\begin{proof}
%	(\ref{eq:elliptic1}) can be reduced to (\ref{eq:elliptic2}). When $p\geq 2$, existence and uniqueness are from the Legendre-Hadamard condition and the Lax-Milgram theorem. The $L^p$ estimate follows from Theorem 10.15 in \cite{Giusti03}.
%	\par When $p\in (1,2)$, it can be done by approximation by $W^{1,2}$ maps and duality. See Proposition 2.11 in \cite{KriGm19}.
%	\par The second one follows from (10.60)-(10.63) on pp.369-370 in \cite{Giusti03}.
%	\end{proof}
	
\section{Auxiliary results for the integrands}\label{sec:aux}
	\par The first two subsections are devoted for estimates of the integrands involved in our proof. Some proofs are omitted and we refer to \cite{KriGm19} for details. Two corollaries of the quasiconvexity of $F$ are given in the third subsection.

\subsection{Estimates for the reference integrand}\label{subsec:refint}
	\par We show some properties for the reference integrand $E_p$ that will be useful later. In the following, only the case $p\in [1,2)$ is considered.
	\par Obviously, we have that $E_p(z)$ is $C^2$, and an elementary calculation gives 
	\begin{align}
	&E_p^{\prime}(w) z = p\langle w \rangle^{p-2}w \cdot z,\\
	&E_p^{\pprime}(w)[z,z] = p\langle w\rangle^{p-4}(\langle w\rangle^2\lvert z\rvert^2+(p-2)\lvert w\cdot z\rvert^2). 
	\end{align} Considering the two cases $p\in (1,2)$ and $p=1$ separately, we have 
	\begin{equation} \label{eq:Ediffest}
	E_p^{\pprime}(w)[z,z] \geq \left\{\ 
		\begin{aligned}
		& p(p-1)\langle w\rangle^{p-2}\abs{z}^2,& &p\in (1,2)\\
		&\langle w\rangle^{-3}\abs{z}^2,& &p=1.
		\end{aligned}\right. 
	\end{equation} Thus, the function $E_p$ is a convex function. In the following, we only consider $E_p$ with $1\leq p<2$. By the definition and convexity of $E_p$, it is easy to get the following:
	\begin{lem}\label{lem:Eest}
	Suppose that $1\leq p<2$ and set $a_1 = \sqrt{2}-1, a_2 =1$. Then the following holds 
		\begin{align}
		& a_1 \min\{\lvert z\rvert^p,\lvert z\rvert^2\} \leq E_p(z) \leq a_2 \min\{\lvert z\rvert^p,\lvert z\rvert^2\}, \label{eq:Eest1}\\
		& E_p(az) \leq \max\{a,a^2\}E_p(z) \quad \mbox{and}\quad E_p(z+w) \leq 2(E_p(z)+E_p(w)) \label{eq:Eest2}
		\end{align} for any $a>0$ and any $z,w \in \mathbb{H}$.
	\end{lem} A corollary of (\ref{eq:Eest1}) is 
		\begin{equation}\label{eq:Eest3}
		\abs{z}^p \leq 1+\frac{1}{a_1}E_p(z), \quad \mbox{for any }z\in \HH,\ p\in [1,2).
		\end{equation}
%	\begin{proof}
%When $t>1$, we compare $E(t)$ and $t$. \[E(t) =\sqrt{1+t^2}-1 \leq 1+t-1 =t,\] then we have the right inequality. \[ 1+t^2-(1+(\sqrt{2}-1)t)^2 =2(\sqrt{2}-1)t(t-1)>0,\] and thus $E(t)-(\sqrt{2}-1)t =\sqrt{1+t^2}-(1+(\sqrt{2}-1)t)>0$. When $t\leq 1$, consider $E(t)-t^2$ and $E(t)-(\sqrt{2}-1)t^2$ and differentiate them to get the desired result. 
%	\par To show (\ref{eq:Eest2}), consider $f(t)=(1+a^2t)^{\frac{1}{2}}-a^2(1+t)^{\frac{1}{2}}+a^2-1, t\geq 0$. %\[ f^{\prime}(t) = a^2\frac{1}{2}((1+a^2t)^{-\frac{1}{2}} -(1+t)^{-\frac{1}{2}}) \leq 0. \] 
%It is easy to see that $f^{\prime}(t)\leq 0$ and thus $f(t) \leq f(0)=0$ for any $t\geq 0$. In particular, $E(az)-a^2E(z) = f(\sqrt{\lvert z\rvert})\leq 0.$
% The second estimate in (\ref{eq:Eest2}) follows easily from the convexity of $E$: \[ E(z+w) = E(\frac{1}{2}(2z+2w)) \leq \frac{1}{2}(E(2z)+E(2w)) \leq 2(E(z)+E(w)). \]
%	\end{proof}
	\begin{lem}\label{lem:Emeanest1}
	Let $1\leq p<2$, $B \subset \R^n$ be an open ball and $u \in L^p(B,\HH)$. Then for any $z\in \mathbb{H}$ we have 
		\begin{equation}\label{eq:Emeanest1}
		\int_B E_p(u-u_B)\id x \leq 4\int_B E_p(u-z)\id x.
		\end{equation} When $p=1$, the function $u$ can be replaced by a bounded $\HH$-valued Radon measure, and the inequality holds in the relaxed sense as in \textnormal{(\ref{eq:flrelaxed})}.
	\end{lem} It is easy to show this lemma for $L^p$ maps with (\ref{eq:Eest2}) and Jensen's inequality, and the estimate for Radon measures follows by mollification.
%	\begin{proof}
%	We extend $\mu$ to $\R^n\setminus B$ by $0$, i.e., for any Borel set $A \subset \R^n$ define $\mu(A):= \mu(A\cap B)$. Let $\{\phi_{\varepsilon}\}$ be the standard mollifiers and set $\mu_{\varepsilon}=\phi_{\varepsilon}\ast \mu$. Then $\mu_{\varepsilon}\in L^1(B,\mathbb{H})$. By (\ref{eq:Eest2}) and Jensen's inequality we have 
%		\begin{align}
%		\int_B E(\mu_{\varepsilon}-(\mu_{\varepsilon})_B) \id x &\leq \int_B 2E(\mu_{\varepsilon}-z) \id x +\Le^n(B) E((\mu_{\varepsilon})_B-z)\\
%		&\leq 4\int_B E(\mu_{\varepsilon}-z)\id x.
%		\end{align} Lemma \ref{lem:areaapprox} implies that $\mu_{\varepsilon}$ converges to $\mu$ area-strictly and we can take $\varepsilon\to 0$ to obtain (\ref{eq:Emeanest1}).
%	\end{proof}
	\begin{lem}\label{lem:Emeanest2}
	Let $1\leq p<2$, $B \subset \R^n$ be an open ball and $f\in L^p(B,\HH)$. Set $\mathcal{E}:= \fint_B E_p(f)\id x$, then we have
		\begin{equation}
		\fint_B \abs{f}^p\id x \leq \sqrt{\mathcal{E}^2+2\mathcal{E}}.
		\end{equation} When $\mathcal{E}\leq a$, it is obvious that the right-hand side can be replaced by $\sqrt{(2+a)\mathcal{E}}$. When $p=1$, we have the analogue holds for bounded $\HH$-valued Radon measures.
	\end{lem} The above lemma gives the estimate of $\fint_B \abs{f}^p\id x$ in terms of $\fint_B E_p(f)\id x$, and can be shown by Jensen's inequality and taking the inverse of $E$.
	\par By definition, we know that for $E_p^A, A\in \R^{N\times n}$, the analogues of Lemma \ref{lem:Eest} and \ref{lem:Emeanest1} hold. Moreover, for any $p\in [1,2)$ there exists $c=c(p)>0$ such that
		\begin{equation}\label{eq:refcomp}
		\frac{1}{c}E_p^A(z) \leq \frac{\lvert z\rvert^2}{(1+\abs{A}+\abs{z})^{2-p}} \leq cE_p^A(z), \quad \mbox{for any }z\in \RNn.
		\end{equation}
%	\begin{proof}
%	With the help of mollification, it is enough to show the case for $L^1(B,\mathbb{H})$ maps. Suppose that $f\in L^1(B,\mathbb{H})$, then Jensen's inequality implies \[ E\left(\fint_B \lvert f\rvert\id x \right) \leq \fint_B E(f)\id x :=\mathcal{E}_f. \] Notice that $E(t) = \sqrt{1+t^2}-1$, which can be written as $t = \sqrt{E^2(t)+2E(t)}$, and thus we have the desired result.
%	\end{proof}

\subsection{Estimates for the shifted integrand}\label{subsec:shifted}
	\par Given any $C^2$ function $F\colon \R^{N\times n}\to \R$, we define for any $w \in \R^{N\times n}$ the corresponding shifted integrand 
	\begin{align}
	F_w(z) & := F(z+w)-F(w)-F^{\prime}(w) z \notag \\
	& = \int_0^1(1-t)F^{\pprime}(w+tz)[z,z]\id t.
	\end{align} 
	\begin{lem}\label{lem:LH}
	Suppose that $F\colon \R^{N\times n} \to \R$ is $C^2$ and satisfies \textnormal{\ref{it:intqc}}. When $p\in (1,2)$, there holds, with $c=c(p)>0$,
		\begin{align}
		&\int_{B} F_w(\nabla \varphi)\id x \geq c\,\ell \int_B E_p^w(\nabla \varphi)\id x,  \label{eq:pshiftedqc} \\
		&F^{\pprime}(w)[\eta \otimes \xi,\eta \otimes \xi] \geq c\,\ell\, \frac{\lvert \eta\rvert^2\lvert \xi\rvert^2}{\langle w \rangle^{2-p}}\label{eq:pLH}
		\end{align} for any ball $B\subset \R^n$, $w\in \R^{N\times n}$, $\varphi \in W^{1,p}_0(B,\R^N)$, $\eta \in \R^N$ and $\xi \in \R^n$. For $p=1$, the corresponding estimates are, with $C>0$,
		\begin{align}
		&\int_{B} F_w(\nabla \varphi)\id x \geq C\ell\int_B \langle w \rangle^{-3} E(\nabla \varphi)\id x ,\label{eq:shiftedqc}\\
		&F^{\pprime}(w)[\eta \otimes \xi,\eta \otimes \xi] \geq C\ell\, \frac{\lvert \eta\rvert^2\lvert \xi\rvert^2}{\langle w\rangle^3} \label{eq:LH}.
		\end{align}
	\end{lem}
	\par The first estimate (\ref{eq:pshiftedqc}) can be showed with the quasiconvexity condition \ref{it:intqc}, \cite{CFM98}, Lemma 2.1 and (\ref{eq:refcomp}). See \cite{KriGm19}, Lemma 4.1 for (\ref{eq:shiftedqc}). The Legendre-Hadamard estimates (\ref{eq:pLH}) and (\ref{eq:LH}) follow from the convexity of $E_p$ and \cite{Federer69}, 5.1.10.

	\begin{lem}
	Suppose that $F\colon \R^{N\times n} \to\R$ satisfies \textnormal{\ref{it:intqc},\ref{it:intreg}} with $p\in [1,2)$. Then for any $m>0$ and any $w\in\R^{N\times n}$ satisfying $\lvert w\rvert\leq m$, we have
		\begin{equation}
		\lvert F_w(z)\rvert \leq CE_p(z)\label{eq:shiftedest2}
		\end{equation} hold for $C=C(m,n,N,L,p)>0$. If we assume \textnormal{\ref{it:int2deriv}} alternatively with $p\in (1,2)$, the estimate becomes
		\begin{equation}\label{eq:ushiftedest1}
		\lvert F_w(z)\rvert\leq CE_p^w(z)
		\end{equation} with $C=C(L,p)>0$.
	\end{lem}
	\begin{proof}
	The estimate (\ref{eq:shiftedest2}) can be obtained with direct calculation and Lemma \ref{lem:intLip} by considering the cases $\lvert z\rvert\leq 1$ and $\lvert z\rvert>1$ separately.
	\par For (\ref{eq:ushiftedest1}), by definition, Taylor's theorem and \ref{it:int2deriv}, we have the estimate
		\begin{align*}
		\lvert F_w(z)\rvert &= \lvert F(z+w)-F(w)-F^{\prime}(w)z\rvert\\
		&=\lvert \int_0^1 F^{\pprime}(w+tz)(1-t)\id t[z,z]\rvert\\
		&\leq L\int_0^1 \frac{1-t}{(1+ \abs{w+tz})^{2-p}}\id t \abs{z}^2.
		\end{align*} Lemma 2.1 in \cite{CFM98} implies that the integral in the last line is controlled by $C(p)(1+\abs{w}+\abs{z})^{p-2}$. The estimate (\ref{eq:refcomp}) then gives the desired result.
	\end{proof}

\begin{lem}\label{lem:shiftedest3}
Suppose that $F\colon \R^{N\times n} \to\R$ satisfies \textnormal{\ref{it:intqc},\ref{it:intreg}} with $p\in [1,2)$. Then for any $m>0$ and any $w\in\R^{N\times n}$ with $\lvert w\rvert\leq m$, there exists a constant $C=C(m,n,N,L,p)>0$ such that
	\begin{equation}\label{eq:shiftedest3}
	\lvert F^{\pprime}_w(0)z-F^{\prime}_w(z)\rvert \leq CE(z).
	\end{equation} Alternatively, with \textnormal{\ref{it:int2deriv}, \ref{it:int2deLip}} and no bound for $w$, we have 
	\begin{equation} \label{eq:ushiftedest2}
		\lvert F^{\pprime}_w(0)z-F^{\prime}_w(z)\rvert \leq C (1+\abs{w})^{p-2}E^w(z)
	\end{equation} with $C=C(n,N,L,p)>0$.
\end{lem}
\begin{proof}
The estimate (\ref{eq:shiftedest3}) can be easily obtained by considering the two cases separately. When $\lvert z\rvert\leq 1$, there holds
	\begin{align*}
	&\ \lvert F^{\pprime}_w(0)z-F^{\prime}_w(z)\rvert = \lvert F^{\pprime}(w)z-(F^{\prime}(w+z)-F^{\prime}(w))\rvert\\
	 =&\ \abs{ \int_0^1(F^{\pprime}(w)-F^{\pprime}(w+tz))\cdot z \id t} \leq C\int_0^1 t\lvert z\rvert^2\id t \leq CE(z),
	\end{align*} where the last line is from \ref{it:intreg} and that $w+tz,w\in B(0,m+1)$. In the other case, we estimate the three terms directly with Lemma \ref{lem:intLip}:
	\begin{align}
	&\ \lvert F^{\pprime}_w(0)z-F^{\prime}_w(z)\rvert = \lvert F^{\pprime}(w)z-(F^{\prime}(w+z)-F^{\prime}(w))\rvert \label{eq:shiftedmid1}\\
	\leq &\ C(m)\lvert z\rvert+CL(2+\abs{w+tz}^{p-1}+\abs{z}^{p-1}) \overset{\lvert z\rvert>1}{\leq} C(m,L)\lvert z\rvert \leq CE(z).\notag
	\end{align} 
	\par For (\ref{eq:ushiftedest2}), the proof is in a similar manner. When $\abs{z}\leq 1+\abs{w}$, the condition \ref{it:int2deLip} implies
	\begin{align*}
	 \abs{F^{\pprime}_w(0)z-F^{\prime}_w(z)} &= \abs{ \int_0^1 (F^{\pprime}(w)-F^{\pprime}(w+tz))\cdot z \id t} \\
	&\leq L\int_0^1 \frac{t\abs{z}^2}{(1+\abs{w}+\abs{w+tz})^{3-p}} \id t\\
	&\overset{p<2}{\leq} L (1+\abs{w})^{p-1} \frac{\abs{z}^2}{(1+\abs{w})^2} \\
	& \leq C(1+\abs{w})^{p-2}E^w(z).
	\end{align*} When $\abs{z}>1+\abs{w}$, the estimate can be obtained in a way similar to (\ref{eq:shiftedmid1}) with Lemma \ref{lem:intLip}.
\end{proof}

\subsection{Local Lipschitz continuity and mean coercivity}
	\par In this subsection, we state two corollaries of the quasiconvexity of $F$. One is the local Lipschitz continuity of $F$ and the other is its $L^p$ mean coercivity.	
	\par It is well-known that separately convex functions are locally Lipschitz (see \cite{Morrey66}, p.112, and \cite{Fusco80, Marcellini85}). Lemma 2.2 in \cite{BKK00} gives a better estimate constant. As a corollary of the above results, the following lemma gives the growth of the derivative of a quasiconvex integrand.
	\begin{lem}\label{lem:intLip}
	Suppose that $G\colon \R^{N\times n}\to\R$ is a real-valued function and of $p$-growth with $p\in [1,\infty)$, i.e., \[ \lvert G(z)\rvert\leq L(1+\lvert z\rvert^p) \] for some $L>0$ and any $z\in \R^{N\times n}$. If $G$ is furthermore quasiconvex, then there exists a constant $C=C(n,N,p)>0$ such that 
		\begin{equation}\label{eq:intLip}
		\lvert G^{\prime}(z)\rvert\leq CL(1+\lvert z\rvert^{p-1}).
		\end{equation} In particular, $G$ is Lipschitz when $p=1$.
	\end{lem} 
	\par The next result is the $L^p$ mean coercivity of $F$, which helps us control the $L^p$-integral of $\lvert \nabla v\rvert$ for $v\in W^{1,p}$ by $\int F(\nabla v)\id x$. For a thorough discussion of the connection between coercivity and quasiconvexity, see \cite{ChenKris17}.
	\begin{lem}\label{lem:1coercive}
	Suppose that $F\colon \R^{N\times n}\to \R$ satisfies \textnormal{\ref{it:intp}} and \textnormal{\ref{it:intqc}} with $p\in [1,2)$. Fix a ball $B_R=B(x_0,R)\subset \R^n$ and $u\in W^{1,p}(B_R,\R^N)$, then there exist $a_3=a_3(p,\ell), a_4=a_4(n,N,L,\ell,p,F) \in \R, a_5 = a_5(n,N,L,\ell,p)>0$ such that 
		\begin{equation}\label{eq:1coercive}
		 a_3\fint_{B_R}\lvert \nabla v\rvert^p\id x+ a_4 \leq \fint_{B_R} F(\nabla v)\id x +a_5\fint_{B_R}\abs{\nabla u}^p\id x
		\end{equation} for any $v\in W^{1,p}_u(B_R,\R^N)$.
	\end{lem}
	\begin{proof}
	First, with the triangle inequality and (\ref{eq:Eest3}) we have 
		\begin{align}
		\fint_{B_R} \abs{\nabla v}^p \id x &\leq C_p \fint_{B_R}(\abs{\nabla v-\nabla u}^p +\abs{\nabla u}^p)\id x \label{eq:coermid1}\\
		&\leq C_p \fint_{B_R}(1+a_1^{-1}E_p(\nabla v-\nabla u) +\abs{\nabla u}^p)\id x.\notag
		\end{align} Notice that $v-u \in W^{1,p}_0(B_R,\R^N)$, then \ref{it:intqc} implies 
		\begin{equation}\label{eq:coermid2}
		\ell \fint_{B_R}E_p(\nabla v-\nabla u)\id x \leq \fint_{B_R}F(\nabla v-\nabla u)\id x -F(0).
		\end{equation} To estimate the integral on the right-hand side, we apply Lemma \ref{lem:intLip} to get
		\begin{align}
		& \abs{\fint_{B_R}(F(\nabla v-\nabla u)-F(\nabla v))\id x} \leq \fint_{B_R}\int_0^1 \abs{F^{\prime}(\nabla v-t\nabla u)} \abs{\nabla u}\id t\id x \label{eq:coermid3}\\
		&\leq CL\fint_{B_R}\int_0^1 (1+\abs{\nabla v-t\nabla u}^{p-1})\abs{\nabla u}\id t\id x \notag\\
		&\leq CL\int_{B_R}(\abs{\nabla u}+\abs{\nabla u}^p +\abs{\nabla u}\abs{\nabla v}^{p-1})\id x \notag\\
		&\leq CL \fint_{B_R}(1+(1+\sigma^{1-p})\abs{\nabla u}^p +\sigma\abs{\nabla v}^p)\id x,\notag
		\end{align} where the $\sigma$ is to be determined. Combining (\ref{eq:coermid1})-(\ref{eq:coermid3}), we know that 
		\begin{align*}
		\fint_{B_R} \abs{\nabla v}^p \id x &\leq c_1 \fint_{B_R}(1+(1+\sigma^{1-p})\abs{\nabla u}^p +\sigma \abs{\nabla v}^p)\id x\\ 
		&\hspace{.6cm} +c_2\fint_{B_R}(F(\nabla v)-F(0))\id x+C_p,
		\end{align*} where $c_1=c_1(n,N,p,L,\ell)>0, c_2=c_2(p,\ell)>0$. Take $\sigma = \frac{1}{2c_1}$, then (\ref{eq:1coercive}) follows. 
	\end{proof}
	\begin{rmk}\label{rmk:1coercive}
	The convexity of $\lvert \cdot\rvert$ together with Remark \ref{rmk:intapprox} tells us that (\ref{eq:1coercive}) can be extended to maps in $BV_u(B_R,\R^N)$ if $p=1$. 
	\end{rmk}
	
\section{Partial regularity for $Du$}\label{sec:partialreg}
	\par This section is for the proof of Theorem \ref{thm:1partialreg}. The function $F$ is assumed to grow linearly near $\infty$ (i.e., $p=1$) and we consider $BV$ $\omega$-minimizers.
\subsection{Caccioppoli-type inequality}
	\par We now give a Caccioppoli-type inequality of the second kind, which is a modified version of Proposition 4.3 in \cite{KriGm19}.
	\begin{prop}\label{prop:1caccioppoli}
	Suppose that $F\colon \R^{N\times n}\to \R$ satisfies \textnormal{\ref{it:intp}-\ref{it:intreg}} with $p=1$, and $u \in BV_{loc}(\Omega,\R^N)$ is an $\omega$-minimizer of $\bar{\F}$ with constant $R_0>0$, where $\omega$ satisfies \textnormal{\ref{it:omega1}}. Then for any $m>0$, there exists $c=c(m,n,N,L, \ell)\geq 1$ such that the following holds: for any $B(x_0,R) \ssubset \Omega$ with $R<R_0$ and any affine map $a\colon \R^n\to \R^N$ satisfying $\lvert \nabla a\rvert\leq m$, we have 
		\begin{equation}\label{eq:1caccioppoli}
		\int_{B_{\frac{R}{2}}} E(D(u-a)) \leq c \left(\int_{B_R} E\left(\frac{u-a}{R}\right)\id x + \omega(R)R^n \right).
		\end{equation}
	\end{prop}
	
	\begin{proof}
	Let $\tilde{F}=F_{\nabla a}$, $\tilde{u}=u-a$. Then $\tilde{u}$ is an $\omega$-minimizer of the relaxed functional corresponding to $\tilde{F}$. Fix $\frac{R}{2}<t<s<R$. Take a smooth cut-off function $\rho$ between $B_t$ and $B_s$ with $\rho \in C_c^{\infty}(B_s)$ and $\lvert \nabla \rho\rvert\leq \frac{2}{s-t}$, and set $\varphi = \rho\tilde{u}, \psi = (1-\rho)\tilde{u}$. Let $\{\phi_{\varepsilon}\}$ be the standard mollifiers and $\varphi_{\varepsilon} = \varphi\ast \phi_{\varepsilon}$, then $\varphi_{\varepsilon} \in W^{1,1}_0(B_s,\R^N)$ when $\varepsilon <\mbox{dist}(\mbox{supp}(\rho),\partial B_s)$.
	\par The strong quasiconvexity of $F$ gives, as in (\ref{eq:shiftedqc}),  \[ C(m,\ell)\int_{B_s} E(\nabla \varphi_{\varepsilon}) \id x \leq \int_{B_s}\tilde{F}(\nabla \varphi_{\varepsilon})\id x. \] Take $\varepsilon \to 0$, then \[ C\int_{B_s} E(D \varphi)\leq \int_{B_s}\tilde{F}(D \varphi). \] We can further proceed as follows:
		\begin{align*}
		&\ C\int_{B_t} E(D\tilde{u}) \leq C\int_{B_s}E(D\varphi) \leq \int_{B_s} \tilde{F}(D\varphi)\\
		=&\ \int_{B_s}\tilde{F}(D\tilde{u})+\int_{B_s} \tilde{F}(D\varphi)-\int_{B_s}\tilde{F}(D\tilde{u})\\
		\leq &\ \int_{B_s}\tilde{F}(D\psi) + \omega(s)\int_{B_s}(1+\lvert D\psi\rvert ) +\int_{B_s} \tilde{F}(D\varphi)-\int_{B_s}\tilde{F}(D\tilde{u})\\
		\overset{(\ref{eq:shiftedest2})}{\leq}&\ C\int_{B_s}E(D\psi) + \omega(s)\int_{B_s}(1+\lvert D\psi \rvert)+C\int_{B_s\setminus B_t}(E(D\varphi)+E(D\tilde{u})).
		\end{align*}The second term can be estimated by the triangle inequality and (\ref{eq:Eest3}): \[
		\omega(s)\int_{B_s}(1+\lvert D\psi \rvert) \leq\omega(s) \left(\omega_n s^n + \int_{B_s}\lvert D\psi\rvert \right) \leq 2\omega(s)\omega_ns^n+C\int_{B_s}E(D\psi).\] Inserting this into the estimate of $C\int_{B_t}E(D\tilde{u})$, we obtain
		\begin{align*}
		&\ \int_{B_t} E(D\tilde{u}) \leq C\int_{B_s}E(D\psi) + \int_{B_s\setminus B_t}(E(D\varphi)+E(D\tilde{u})) + C\omega(s)s^n\\
		=&\ \int_{B_s} E((1-\rho)D\tilde{u}-\tilde{u}\otimes \nabla \rho) + \int_{B_s\setminus B_t} (E(D\tilde{u})+E(\rho D\tilde{u}+\tilde{u}\otimes \nabla\rho))+C\omega(s)s^n\\
		\leq&\ C\int_{B_s\setminus B_t}E(D\tilde{u}) +C\int_{B_s}E\left(\frac{\tilde{u}}{s-t}\right)\id x + C\omega(R)R^n.
		\end{align*} Now we can apply the hole-filling trick, adding $C\int_{B_t}E(D\tilde{u})$ to both sides, and divide the inequality by $C+1$. Finally, by the following iteration lemma we have the desired inequality.
	\end{proof}
	\begin{lem}\label{lem:iteraineq}
Suppose that $\theta \in (0,1), R>0$ and the two functions $\Phi, \Psi\colon (0,R]\to \R$ are positive. $\Phi$ is bounded, and $\Psi$ is decreasing with $\Psi(\sigma \rho) \leq \sigma^{-2}\Psi(\rho)$ for any $\rho\in (0,R], \sigma\in (0,1]$. If for any $\frac{R}{2}\leq t<s \leq R$ there holds
	\begin{equation}\label{eq:iteraineq}
	\Phi(t) \leq \theta\, \Phi(s)+\Psi(s-t)+B
	\end{equation} for some $B>0$, then we have, for some $C=C(\theta)>0$,
	\begin{equation}
	\Phi\left(\frac{R}{2}\right) \leq C(\Psi(R)+B).
	\end{equation}
\end{lem} This lemma is widely used in the proofs of Caccioppoli-type inequalities and can be shown by modifying Lemma 6.1 in \cite{Giusti03}.
%\begin{proof}
%Take $\lambda \in (0,1)$ to be determined later, and set \[ r_0 =\frac{R}{2},\qquad r_{i+1}-r_i = (1-\lambda)\lambda^i\frac{R}{2}. \] Then apply (\ref{eq:iteraineq}) iteratively to obtain
%	\begin{align*}
%	\Phi(r_0) &\leq \theta \Phi(r_1)+\Psi((1-\lambda)\frac{R}{2}) +B\\
%	&\leq \theta (\theta \Phi(r_2)+\Psi((1-\lambda)\lambda\frac{R}{2}) +B)+\Psi((1-\lambda)\frac{R}{2}) +B\\
%	&\leq \dots \leq \theta^{k+1}\Phi(r_{k+1}) +\sum_{i=0}^k \theta^i\left(\Phi((1-\lambda)\lambda^i\frac{R}{2}) +B \right).
%	\end{align*} Since $\theta^i\Phi((1-\lambda)\lambda^i\frac{R}{2}) \leq 4(1-\lambda)^{-2}\lambda^{-2i}\theta^i \Phi(R)$, taking $\sqrt{\theta} <\lambda<1$ we have \[ \Phi(r_0) \leq \theta^{k+1}\Phi(r_{k+1}) + \frac{4}{(1-\lambda)^2} \frac{1-(\theta/\lambda^2)^{k+1}}{1-(\theta/\lambda^2)} \Psi(R) + \frac{1-\theta^{k+1}}{1-\theta}B. \] The desired result can be obtained by taking $k\to \infty$.
%\end{proof} 
%It is obvious that we can replace $\frac{R}{2}$ by any $0<R^{\prime}<R$ and with a similar proof get $\Phi(R^{\prime}) \leq C(\Psi(R-R^{\prime})+B)$ instead.

\subsection{Euler-Lagrange inequality}\label{subsec:ELineq}
	\par The minimizers of regular functionals satisfy the corresponding Euler-Lagrange equations. In the case of $\omega$-minimizers, we do not have such equations hold anymore, while a corresponding inequality can be obtained instead with the help of Ekeland's variational principle(see \cite{Eke74} and \cite{Giusti03}, Theorem 5.6).
	\begin{lem}[Ekeland variational principle]\label{lem:Ekeland} Suppose that $(X,d)$ is a complete metric space and $\mathcal{F}\colon X \to \R\cup \{\infty\}$ is a lower-semicontinuous function with respect to $d$, not identically $\infty$ and has a lower bound. If for some $u\in X$ and $\varepsilon>0$ we have \[ \mathcal{F}(u) \leq \inf_{v\in X}\mathcal{F}(v) +\varepsilon, \] then there exists a $w\in X$ satisfying the following:
		\begin{enumerate}[label=\textup{(\alph{enumi})}]
		\item $d(u,w) \leq \sqrt{\varepsilon}$;
		\item $\mathcal{F}(w) \leq \mathcal{F}(u)$;
		\item $\mathcal{F}(w) \leq \mathcal{F}(v)+\sqrt{\varepsilon}d(v,w)$ for any $v\in X$.
		\end{enumerate}
	\end{lem}
	\par Suppose that $F\colon \RNn \to \R$ satisfies \textnormal{\ref{it:intp}-\ref{it:intreg}} with $p=1$, $\omega$ satisfies \textnormal{\ref{it:omega1}} and $u\in BV_{loc}(\Omega,\R^N)$ is an $\omega$-minimizer of $\bar{\F}$ with constant $R_0>0$. Take $B_R=B(x_0,R)\ssubset \Omega$ such that $R<R_0$, $\lvert Du\rvert(\partial B_R)=0$ and $u\vert _{\partial B_R}\in BV(\partial B_R,\R^N)$, which is possible by Lemma \ref{lem:BVFubini}. For any $\delta>0$, Remark \ref{rmk:intapprox} implies that there exists $u_{\delta}\in W^{1,1}_u(\B_R,\R^N)$ such that 
		\begin{align}
		&\fint_{B_R}\frac{\lvert u-u_{\delta}\rvert}{R}\id x <\delta, \quad \left\lvert \fint_{B_R}E(Du)-\fint_{B_R}E(\nabla u_{\delta})\id x \right\rvert<\delta, \label{eq:1ELmid3}\\
		&\left\lvert \fint_{B_R}F(\nabla u_{\delta})\id x - \fint_{B_R}F(Du) \right\rvert < \delta. \label{eq:1ELmid2}
		\end{align} By the $\omega$-minimality of $u$ we know 
		\begin{equation}\label{eq:1ELmid1}
		\bar{\F}(u,B_R) \leq  \bar{\F}(v,B_R) + \omega(R)\int_{B_R}(1+|Dv|),\quad \mbox{for any } v \in BV_u(B_R,\R^N).
		\end{equation}
	%\red{If we take the alternative definition of $\omega$-minimizers, the second term on the RHS should be \[\omega(R)\int_{B_R}(1+|Du|+|Du-Dv|) \leq \omega(R)\int_{B_R}(1+2|Du|+|Dv|), \] which will not affect the following as the $\varepsilon$ we take contains $\int |Du|$.} 
	Again from Remark \ref{rmk:intapprox}, we know that \[ \inf_{v\in W^{1,1}_u(B_R,\R^N)}\F (v,B_R) = \inf_{v\in BV_u(B_R,\R^N)}\bar{\F}(v,B_R)=:I \] and there exists $\{v_j\}\subset W^{1,1}_v(B_R,\R^N)$ such that $\F(v_j,B_R)\to I$. The  mean coercivity of $F$ (Lemma \ref{lem:1coercive}) implies 
		\begin{align*}
		\fint_{B_R} (1+\lvert Dv_j\rvert)\id x &\leq 1 + \frac{1}{a_3}\brac{\fint_{B_R}F(\nabla v_j)\id x +a_5 \fint_{B_R}\abs{Du}-a_4}\\
		&\leq 1-\frac{a_4}{a_3} +\frac{a_5}{a_3}\fint_{B_R}\abs{Du} + \frac{1}{a_3}\brac{\fint_{B_R} F(Du)+\delta_j}\\
		&\leq \fint_{B_R}\brac{a_6+a_7\abs{Du}} +\frac{\delta_j}{a_3},
		\end{align*} where $\delta_j := \F(v_j,B_R)- I \to 0$ as $j\to\infty$, and $a_6=1+\frac{L-a_4}{a_3}$, $a_7 =\frac{a_5+L}{a_3}$. Take $v$ to be $v_j$ in (\ref{eq:1ELmid1}) and let $j\to \infty$, then we have
		\begin{equation}
		\bar{\F}(u,B_R) \leq \inf_{v\in BV_u(B_R,\R^N)}\bar{\F}(v,B_R) +\omega(R)\int_{B_R}\left(a_6+a_7\lvert Du\rvert\right).
		\end{equation}
	\par Set $\varepsilon = \omega(R)\fint_{B_R}(a_6+a_7\lvert Du\rvert)$. From the above estimate of $\int (1+\lvert Dv_j\rvert)\id x$ we can see $\varepsilon >0$. Then by (\ref{eq:1ELmid2}) we have 
		\begin{align*}
		\mathscr{F}(u_{\delta},B_R) &\leq \inf_{v \in BV_u(B_R,\R^N)}\bar{\mathscr{F}}(v,B_R)+\omega_nR^n(\varepsilon+\delta)\\
		 &= \inf_{v \in W^{1,1}_u(B_R,\R^N)}\mathscr{F}(v,B_R)+\omega_nR^n(\varepsilon+\delta),
		\end{align*} where $\omega_n = \Le^n(B(0,1))$ is the Lebesgue measure of the unit ball in $\R^n$. Consider the complete metric space $X = W^{1,1}_u(B_R,\R^N)$ with $d(w_1,w_2) = \fint_{B_R}\lvert \nabla (w_1-w_2)\rvert\id x$. To apply the Ekeland variational principle Lemma \ref{lem:Ekeland}, we take $\mathcal{F}(w) = \fint_{B_R}F(\nabla w)\id x$ and replace $\varepsilon$ by $\varepsilon+\delta$. Then there is $w \in W^{1,1}_u(B_R,\R^N)$ such that
		\begin{enumerate}[label=(\alph*)]
		\item\label{it:1Ekeland1} $d(u_{\delta},w) \leq \sqrt{\varepsilon+\delta}$;
		\item $\mathcal{F}(w) \leq \mathcal{F}(u_{\delta})$;
		\item\label{it:1Ekeland3} $\mathcal{F}(w) \leq \mathcal{F}(v) + \sqrt{\varepsilon+\delta}\,d(w,v)$, for any $v \in X=W^{1,1}_u(B_R,\R^N)$.
		\end{enumerate} For any $\varphi \in W^{1,1}_0(B_R,\R^N)$, we take $v=w+t\varphi$ and insert it into \ref{it:1Ekeland3} to obtain \[ \fint_{B_R}F(\nabla w)\id x \leq \fint_{B_R}F(\nabla (w+t\varphi))\id x + \sqrt{\varepsilon + \delta} \lvert t\rvert \fint_{B_R}\lvert \nabla \varphi\rvert\id x. \] Differentiate with respect to $t$, then we have an Euler-Lagrange inequality
		\begin{equation}\label{eq:1ELineq}
		\left\lvert \fint_{B_R}F^{\prime}(\nabla w) \nabla \varphi \id x \right\rvert \leq \sqrt{\varepsilon + \delta}\fint_{B_R}\lvert \nabla \varphi\rvert\id x.
		\end{equation}

\subsection{Harmonic approximation}\label{subsec:harapprox}
	\par Now we compare $u$ with a harmonic map $h$ which coincides with it on the boundary of a certain ball. With the estimate of $u-h$, we are able to transfer some regularity of $h$ to $u$.
	\begin{prop}\label{prop:1harapprox}
	Suppose that the function $F\colon \RNn \to \R$ satisfies \textnormal{\ref{it:intp}-\ref{it:intreg}} with $p=1$, $\omega$ satisfies \textnormal{\ref{it:omega1}} and $u\in BV_{loc}(\Omega,\R^N)$ is an $\omega$-minimizer of $\bar{\F}$ with constant $R_0>0$. Let $m>0$ be a fixed constant, $a\colon \R^n \to \R^N$ be affine with $\abs{\nabla a}\leq m$ and $\tilde{F}:=F_{\nabla a}$. Assume that $B_R = B(x_0,R) \ssubset \Omega$ is a ball such that $\lvert Du\rvert(\partial B_R)=0$ and $u\vert_{\partial B_R} \in BV(\partial B_R,\R^N)$. Then the system 
		\begin{equation}\label{eq:1ellipsystem}
		\left\{\begin{aligned}
		&-\mathrm{div}(\tilde{F}^{\pprime}(0)\nabla h) = 0,& &\mathrm{in }\ B_R\\
		&h\vert_{\partial B_R}=u\vert_{\partial B_R},& &\mathrm{on }\ \partial B_R
		\end{aligned} \right.
		\end{equation} admits a unique solution $h \in W^{1,r}(B_R,\R^N)$ such that 
		\begin{equation}\label{eq:1harapproxest}
		\left(\fint_{B_R}\lvert \nabla h-\nabla a\rvert^r\id x \right)^{\frac{1}{r}} \leq C \fint_{\partial B_R}\lvert D_{\tau}(u-a)\rvert.
		\end{equation} The exponent $r$ is as in Lemma \textnormal{\ref{lem:BVbdry}} and $C=C(m,n,N,\frac{L}{\ell},r)>0$. Furthermore, for any $q\in (1,\frac{n}{n-1})$, there exists a constant $C=C(m,n,N,L,\ell,q)>0$ such that
		\begin{equation}\label{eq:1harapprox}
		\fint_{B_R} E\left(\frac{u-h}{R}\right)\id x \leq C\left(\fint_{B_R}E(D(u-a)) \right)^{q}+C(\sqrt{\varepsilon}+\sqrt{\varepsilon}^q),
		\end{equation} where $\varepsilon = \omega(R)\fint_{B_R}(a_6+a_7\lvert Du\rvert)$ is as in last subsection. 
	\end{prop}
	\begin{proof}
	From \ref{it:intreg} and (\ref{eq:LH}) we have that $\lvert \tilde{F}^{\pprime}(0)\rvert \leq C(m)$ and satisfies the Legendre-Hadamard condition. By Lemma \ref{lem:BVbdry} we know that $u\vert_{\partial B_R} \in W^{1-\frac{1}{r},r}(\partial B_R,\R^N)$ for a proper $r$ and \[\left(\fint_{\partial B_R}\int_{\partial B_R} \frac{\lvert u(x)-u(y)\rvert^r}{\lvert x-y\rvert^{n+r-2}}\id \mathcal{H}^{n-1}(x)\id \mathcal{H}^{n-1}(y)\right)^{\frac{1}{r}} \leq CR^{\frac{1}{r}}\fint_{\partial B_R}\lvert D_{\tau}u\rvert.\] Lemma \ref{lem:ellip2} implies the existence of a unique solution $h\in W^{1,r}(B_R,\R^N)$ to (\ref{eq:1ellipsystem}). By replacing $u$ by $\tilde{u}$, we have the estimate (\ref{eq:1harapproxest}).
	 \par Let $\delta, u_{\delta}$ and $w$ be as in last subsection. Take an arbitrary $\varphi \in C_c^{\infty}(B_R,\R^N)$, then we have 
		\begin{align}
		&\quad \fint_{B_R} \tilde{F}^{\pprime}(0)[\nabla (w-h),\nabla \varphi]\id x= \fint_{B_R} \tilde{F}^{\pprime}(0)[\nabla \tilde{w},\nabla \varphi]\id x \notag\\
		&= \fint_{B_R}(\tilde{F}^{\pprime}(0)[\nabla \tilde{w},\nabla \varphi]-\tilde{F}^{\prime}(\nabla \tilde{w})\nabla \varphi) \id x + \fint_{B_R} F^{\prime}(\nabla w) \nabla \varphi \id x \label{eq:1harmid1}\\
		&\leq C\fint_{B_R}E(\nabla \tilde{w})\lvert \nabla\varphi\rvert\id x +\sqrt{\varepsilon+\delta}\fint_{B_R}\lvert \nabla \varphi\rvert\id x,\notag
		\end{align} where $\tilde{w}=w-a$ and the last line is obtained with Lemma \ref{lem:shiftedest3} and (\ref{eq:1ELineq}). By approximation, $\varphi$ can be taken in $W^{1,\infty}_0\cap C^1(B_R,\R^N)$. To obtain the desired estimate, we need to find a proper test map $\varphi$, before which we scale to the unit ball $B(0,1) (=:B)$. 
	\par Define $\psi:=w-h$, and set \[ \Psi(y):= \frac{\psi(x_0+Ry)}{R},\quad \Phi(y):=\frac{\varphi(x_0+Ry)}{R},\quad \tilde{W}(y):=\frac{\tilde{w}(x_0+Ry)}{R}. \] Consider the system, with $\A:= \tilde{F}^{\pprime}(0)$,
		\begin{equation}\label{eq:1dualelliptic}
		\left\{\begin{aligned}
		&-\mbox{div}(\A\nabla \Phi) = T(\Psi),& &\mbox{in }B\\
		&\Phi\vert_{\partial B} = 0,& &\mbox{on }\partial B,
		\end{aligned} \right.
		\end{equation} where 
		\begin{equation}\label{eq:truncation}
		 T(\Psi) := \left\{\begin{aligned}
		&\Psi,& &\lvert \Psi\rvert \leq 1\\
		&\frac{\Psi}{\lvert \Psi\rvert},& &\lvert \Psi\rvert>1.
		\end{aligned} \right. 
		\end{equation} As $T(\Psi) \in L^{\infty}(B_R,\R^N)$, the solution $\Phi$ exists and lies in $W^{1,s}_0\cap W^{2,s}(B_R,\R^N)$ for any $s>1$. We take $s>n$ so that by Morrey's inequality 
		\begin{equation}
		\|\nabla \Phi\|_{L^{\infty}} \leq C\|\Phi\|_{W^{2,s}} \leq C\|T(\Psi)\|_{L^s} \leq C\left(\fint_{B}E(\Psi)\id x \right)^{\frac{1}{s}}.
		\end{equation} Thus, the following can be deduced from (\ref{eq:1harmid1})
		\begin{align*}
		\fint_{B}E(\Psi) \id x &\leq a_2 \fint_{B}\min\{\lvert \Psi\rvert,\lvert \Psi\rvert^2\}\id x = a_2 \fint_{B}[T(\Psi),\Psi]\id x\\
		&=a_2 \fint_{B}\A[\nabla \Phi,\nabla \Psi]\id x = a_2 \fint_{B}\A[\nabla \Psi,\nabla \Phi]\id x\\
		&\overset{(\ref{eq:1harmid1})}{\leq} C \fint_{B}E(\nabla \tilde{W})\lvert \nabla \Phi\rvert\id x + a_2 \sqrt{\varepsilon+\delta}\fint_{B}\lvert \nabla \Phi\rvert\id x\\
		&\leq C\left(\fint_{B}E(\nabla \tilde{W})\id x+\sqrt{\varepsilon+\delta} \right)\left(\fint_{B}E(\Psi)\id x\right)^{\frac{1}{s}}.
		\end{align*} Setting $q=s^{\prime}=\frac{s}{s-1}$, we can obtain 
		\begin{equation}\label{eq:1harmid2}
		\fint_{B}E(\Psi)\id x \leq C\left(\fint_{B}E(\nabla \tilde{W})\id x \right)^q+C(\varepsilon + \delta)^{\frac{q}{2}}.
		\end{equation} Back to $B_R$, the above inequality becomes 
		\begin{equation}\label{eq:1harmid3} 
		\fint_{B_R}E\left(\frac{w-h}{R}\right)\id x \leq C\left(\fint_{B_R}E(\nabla (w-a))\id x \right)^q+C(\varepsilon + \delta)^{\frac{q}{2}}. 
		\end{equation} To compare $u$ and $h$, we decompose $u-h$ as $(u-u_{\delta})+(u_{\delta}-w)+(w-h)$:
		\begin{align*}
		&\ \fint_{B_R}E\left(\frac{u-h}{R}\right)\id x \leq C\fint_{B_R}\left(E\left(\frac{u-u_{\delta}}{R}\right)+E\left(\frac{u_{\delta}-w}{R}\right)+E\left(\frac{w-h}{R}\right) \right)\id x \\
		\leq &\ C\fint_{B_R}\frac{\lvert u-u_{\delta}\rvert}{R}\id x + C\fint_{B_R}\frac{\lvert u_{\delta}-w\rvert}{R}\id x + C\fint_{\B_R}E\left(\frac{w-h}{R}\right)\id x\\
		\leq &\ C\delta +C\fint_{B_R}\lvert \nabla (u_{\delta}-w)\rvert\id x + C\fint_{\B_R}E\left(\frac{w-h}{R}\right)\id x\\
		\leq &\ C\delta + C\sqrt{\varepsilon+\delta} +C\left(\fint_{B_R}E(\nabla (w-a))\id x \right)^q+C(\varepsilon + \delta)^{\frac{q}{2}},
		\end{align*} where the third line comes from (\ref{eq:1ELmid3}) and Poincar\'{e}'s inequality, and the fourth the difference between $u_{\delta}$ and $w$ (see \ref{it:1Ekeland1}) and (\ref{eq:1harmid3}).  The term concerning $w-a$ can be controlled in terms of $u-a$:
		\begin{align*}
		&\ \fint_{B_R} E(\nabla (w-a))\id x \leq C\fint_{B_R}\left(E(\nabla (w-u_{\delta})+E(\nabla \tilde{u}_{\delta})\right)\id x\\
		\leq &\ C\fint_{B_R}\lvert \nabla (w-u_{\delta})\rvert\id x +C\left(\fint_{B_R}E(\nabla\tilde{u}_{\delta})-\fint_{B_R}E(D\tilde{u}) \right)+C\fint_{B_R}E(D\tilde{u})\\
		\leq &\ C\sqrt{\varepsilon+\delta}+C (\delta)+C\fint_{B_R}E(D\tilde{u}),
		\end{align*} where $C(\delta)$ is a $\delta$-related constant and goes to $0$ as $\delta\to 0$ (see Remark \ref{rmk:intapprox}). Combining the estimates above and taking $\delta\to 0$, we have (\ref{eq:1harapprox}) hold.
	\end{proof}
	
\subsection{Excess decay estimate}
	\par For any ball $B(x_0,R)\ssubset \Omega$, define the excess of $u$ as \[ \mathscr{E}(x_0,R) := \fint_{B_R}E(Du-(Du)_{B_R}). \]
	\begin{prop}\label{prop:1excessest}
	Suppose that $F\colon \R^{N\times n} \to \R$ satisfies \textnormal{\ref{it:intp}-\ref{it:intreg}} with $p=1$ and $\omega\colon [0,\infty) \to [0,\infty)$ satisfies \textnormal{\ref{it:omega1}-\ref{it:omega3}}. The map $u\in BV_{loc}(\Omega,\R^N)$ is an $\omega$-minimizers of $\bar{\F}$ with constant $R_0>0$. If the ball $B_R=B(x_0,R) \ssubset \Omega$ with $R<R_0$ is such that
		\begin{equation}
		\lvert (Du)_{B_R}\rvert<m,   \quad \fint_{B_R}\lvert Du-(Du)_{B_R}\rvert \leq 1
		\end{equation} for some $m>0$, then we have 
		\begin{equation}\label{eq:1excessest}
		\mathscr{E}(\sigma R) \leq c (\sigma^2 + \sigma^{-(n+2)}\mathscr{E}(R)^{q-1})\mathscr{E}(R)+c\sigma^{-(n+2)}\sqrt{\omega(R)}
		\end{equation} holds for any $\sigma \in (0,1)$ and any $q\in (1,\frac{n}{n-1})$ with some $c=c(m,n,N,L,\ell,q)>0$.
	\end{prop}
	\begin{proof}
	When $\sigma \geq \frac{1}{5}$, (\ref{eq:1excessest}) is easy to show, and thus we only consider the case $\sigma \in (0,\frac{1}{5})$. Set $a(x) = u_{B_R} +(Du)_{B_R}(x-x_0)$, $\tilde{u}=u-a$ and $\tilde{F}=F_{\nabla a}$. Take $\rho\in (\frac{9}{10}R,R)$ such that $\lvert D\tilde{u}\rvert(\partial B_{\rho})=0$ and $\tilde{u}\vert_{\partial B_{\rho}}\in BV(\partial B_{\rho},\R^N)$, then by Lemma \ref{lem:BVbdry} and \ref{lem:BVFubini}, we have $\tilde{u}\vert_{\partial B_{\rho}} \in W^{1-\frac{1}{r},r}(\partial B_{\rho},\R^N)$ ($r=\frac{n}{n-1}$ if $n\geq 3$, $r\in(1,2)$ if $n=2$), and 
	\begin{equation}\label{eq:1exdecaymid}
	[\tilde{u}\vert_{\partial B_{\rho}}]_{W^{1-\frac{1}{r},r}} \leq C\int_{\partial B_{\rho}}\lvert D_{\tau}(\tilde{u}\vert_{\partial B_{\rho}})\rvert \leq \frac{C}{R}\int_{B_R}\lvert D\tilde{u}\rvert. 
	\end{equation} Let $h$ be the harmonic map determined by (\ref{eq:1ellipsystem}) with $R$ replaced by ${\rho}$. We moreover define \[ \tilde{h}=h-a, \quad a_1(x)=\tilde{h}(x_0)+\nabla \tilde{h}(x_0)(x-x_0),\quad a_0=a+a_1. \] Then Lemma \ref{lem:ellip1}, Remark \ref{rmk:ellipest} and (\ref{eq:1exdecaymid}) imply
		\begin{align*}
		&\ \lvert \nabla \tilde{h}(x_0)\rvert\leq \sup_{B_{\frac{\rho}{2}}} \lvert \nabla \tilde{h}\rvert \leq C\fint_{B_{\rho}}\lvert \nabla \tilde{h}\rvert\id x\leq C\left(\fint_{B_{\rho}}\lvert \nabla \tilde{h}\rvert^r\id x \right)^{\frac{1}{r}}\\
		\leq &\ [\tilde{u}\vert_{\partial B_{\rho}}]_{W^{1-\frac{1}{r},r}}\leq \frac{C}{R{\rho}^{n-1}}\int_{B_R}\lvert D\tilde{u}\rvert \leq C\fint_{B_R}\lvert D\tilde{u}\rvert.
		\end{align*} Then by assumption, it is possible to control $\abs{\nabla a_0}$ as follows \[ \lvert \nabla a_0\rvert \leq \lvert \nabla a\rvert+\lvert \nabla a_1\rvert \leq \lvert (Du)_{B_R}\rvert+C\fint_{B_R}\lvert D\tilde{u}\rvert \leq m+C =:C_m. \]
	\par For any $\sigma \in(0,\frac{1}{5})$, we have $2\sigma R<\frac{\rho}{2}$. Lemma \ref{lem:Emeanest1} gives
		\begin{equation}\label{eq:1exdecaymid1}
		 \fint_{B_{\sigma R}}E(Du-(Du)_{B_{\sigma R}}) \leq 4\fint_{B_{\sigma R}}E(D(u-a_0)), 
		\end{equation} and inequality (\ref{eq:1caccioppoli}) implies 
		\begin{align}
		&\ \fint_{B_{\sigma R}}E(D(u-a_0)) \leq C\left(\fint_{B_{2\sigma R}}E\brac{\frac{u-a_0}{2\sigma R}}\id x +\omega(2\sigma R) \right) \label{eq:1exdecaymid2}\\
		\leq&\ 2C\fint_{B_{2\sigma R}}\left(E\left(\frac{\tilde{u}-\tilde{h}}{2\sigma R}\right)+E\left(\frac{\tilde{h}-a_1}{2\sigma R}\right) \right)\id x+C\omega(2\sigma R).\notag
		\end{align} By Lemma \ref{lem:ellip1} we have, for $x\in B_{2\sigma R}$,
		\begin{align*}
		\frac{\lvert \tilde{h}(x)-a_1(x)\rvert}{2\sigma R} &\leq C\sup_{B_{2\sigma R}}\lvert \nabla^2 h\rvert \frac{\lvert x-x_0\rvert^2}{2\sigma R} \leq C\sigma R \sup_{B_{\frac{\rho}{2}}}\lvert \nabla ^2 \tilde{h}\rvert\\
		&\leq C\sigma \fint_{B_{\rho}}\lvert \nabla \tilde{h}\rvert\id x \leq C\sigma \fint_{\partial B_{\rho}}\lvert D_{\tau}(\tilde{u}\vert_{\partial B_{\rho}})\rvert\\
		&\overset{(\ref{eq:1exdecaymid})}{\leq} C\sigma \fint_{B_R}\lvert Du-(Du)_{B_R}\rvert \\
		&\leq C\sigma \left(\fint_{B_R}E(Du-(Du)_{B_R}) \right)^{\frac{1}{2}}.
		\end{align*} The assumption implies $\fint_{B_R}E(D\tilde{u})\leq \fint_{B_R}\abs{D\tilde{u}}\leq 1$, then Lemma \ref{lem:Emeanest2} can be used to get the last line. Thus, the integral involving $\tilde{h}-a_1$ is controlled by the following
		\begin{align}
		\fint_{B_{2\sigma R}}E\left(\frac{\tilde{h}-a_1}{2\sigma R}\right) \id x &\leq E\left(C\sigma \left(\fint_{B_R}E(Du-(Du)_{B_R}) \right)^{\frac{1}{2}} \right)\label{eq:1exdecaymid3} \\
		&\leq a_2C\sigma^2 \fint_{B_R}E(Du-(Du)_{B_R}). \notag
		\end{align} The term concerning $\tilde{u}-\tilde{h}$ can be estimated with (\ref{eq:1harapprox}): 
		\begin{align}
		&\ \fint_{B_{2\sigma R}} E\left(\frac{\tilde{u}-\tilde{h}}{2\sigma R}\right)\id x \leq \frac{C}{\sigma^{n+2}}\fint_{B_{\rho}}E\left(\frac{\tilde{u}-\tilde{h}}{\rho}\right)\id x  \label{eq:1exdecaymid4} \\
		\leq&\ \frac{C}{\sigma^{n+2}}\left(\left(\fint_{B_{\rho}}E(D(u-a))\right)^q+\sqrt{\varepsilon}+\sqrt{\varepsilon}^q \right), \notag
		\end{align} where $\varepsilon=\fint_{B_{\rho}}(a_6+a_7\abs{Du})$. Considering $\lvert Du\rvert \leq \lvert Du-(Du)_{B_R}\rvert+\lvert (Du)_{B_R}\rvert$, we obtain by assumption that $\varepsilon\leq C\omega(\rho)$. The above estimates (\ref{eq:1exdecaymid1})-(\ref{eq:1exdecaymid4}) and the estimate for $\varepsilon$ together give 
		\begin{align}
		\fint_{B_{\sigma R}}E(Du-(Du)_{B_{\sigma R}}) \leq &\ \frac{C}{\sigma^{n+2}}\left(\left(\fint_{B_R}E(Du-(Du)_{B_R}) \right)^{q}+\sqrt{\omega(R)} \right) \\
		& \hspace{0.6cm}+C\sigma^2\fint_{B_R}E(Du-(Du)_{B_R}) + C\omega(2\sigma R),\notag
		\end{align} which is exactly (\ref{eq:1excessest}).
	\end{proof}

\subsection{Iteration}\label{subsec:1iteration}
	\par Now it is the time to do iteration with (\ref{eq:1excessest}) and get a first regularity result of $u$. Before that we present a lemma concerning the summability of $\omega$ to some power.
	\begin{lem}\label{lem:omegaest1}
	For any fixed $r>0$, $\alpha \in [\frac{1}{4},1)$ and $\tau \in (0,1)$, we have 
		\begin{equation}\label{eq:omegaest1}
		\sum_{j=0}^{\infty} \omega^{\alpha}(\tau^jr)\leq \frac{2\alpha\beta}{1-\tau^{2\alpha\beta}}\Xi_{\alpha}(r),
		\end{equation} where $\beta$ is as in \textup{\ref{it:omega2}}. In particular, 
		\begin{equation}
		\omega^{\alpha}(r)\leq \Xi_{\alpha}(r).
		\end{equation} 
	\end{lem}
	\begin{proof}
	The idea is to transform the sum on the left-hand side into that of a series of integrals on many subintervals of $[0,r]$. Indeed, by \ref{it:omega2}
		\begin{align*}
		\int_{\tau^{j}r}^{\tau^{j-1}r} \frac{\omega^{\alpha}(\rho)}{\rho}\id \rho & \geq \frac{\omega^{\alpha}(\tau^{j-1}r)}{(\tau^{j-1}r)^{2\alpha\beta}}\int_{\tau^jr}^{\tau^{j-1}r} \rho^{2\alpha\beta-1}\id \rho \\
		&=\frac{\omega^{\alpha}(\tau^{j-1}r)}{(\tau^{j-1}r)^{2\alpha\beta}}\frac{1}{2\alpha\beta} ((\tau^{j-1}r)^{2\alpha\beta}-(\tau^jr)^{2\alpha\beta})\\
		&= \frac{1}{2\alpha\beta}(1-\tau^{2\alpha\beta})\,\omega^{\alpha}(\tau^{j-1}r). 
		\end{align*} Summing over $j$ we obtain \[ \sum_{j=0}^{\infty}\omega^{\alpha}(\tau^jr) \leq \frac{2\alpha\beta}{1-\tau^{2\alpha\beta}}\int_0^r \frac{\omega^{\alpha}(\rho)}{\rho}\id \rho = \frac{2\alpha\beta}{1-\tau^{2\alpha\beta}}\Xi_{\alpha}(r). \]
	\end{proof}
	\begin{prop}\label{prop:epsreg1}
	Suppose that $F\colon \R^{N\times n} \to \R$ satisfies \textnormal{\ref{it:intp}-\ref{it:intreg}} with $p=1$, $\omega\colon [0,\infty) \to [0,\infty)$ satisfies \textnormal{\ref{it:omega1}-\ref{it:omega3}} and $u\in BV_{loc}(\Omega,\R^N)$ is an $\omega$-minimizer of $\bar{\F}$ with constant $R_0>0$. For any $\alpha \in (\frac{\beta}{2},1)$ and $m>0$, there exist $C = C(m,n,N,L,\ell,\alpha,\beta)>0,\varepsilon_m>0$ and $R_1>0$ such that the following holds: if $B_R=B(x_0,R) \ssubset \Omega$ is such that 
		\begin{equation}\label{eq:epsregcon}
		\lvert (Du)_{B_R}\rvert<m, \quad \mathscr{E}(x_0,R)<\frac{\varepsilon_m}{2},\quad R<R_1,
		\end{equation} then for any $0<\rho<R$,
		\begin{equation}\label{eq:schauder1}
		\mathscr{E}(\rho) \leq C \left(\frac{\rho}{R} \right)^{2\alpha}\E(R)+C\sqrt{\omega(\rho)}.
		\end{equation}
	\end{prop}
	\begin{proof}
	By Lemma \ref{lem:Emeanest2}, we have \[ \fint_{B_R}\lvert Du-(Du)_{B_R}\rvert \leq \sqrt{3\mathscr{E}(R)} \] if $\mathscr{E}(R)\leq 1$. Then set $\varepsilon_m<\frac{1}{3}$ so that $\fint_{B_R}\lvert Du-(Du)_{B_R}\rvert<1$. Meanwhile, we take $R_1$ such that $\omega(R_1)<1$. The assumptions of Proposition \ref{prop:1excessest} are satisfied and then, for some fixed $q\in (1,\frac{n}{n-1})$, 
		\begin{equation}\label{eq:1iteramid1}
		\mathscr{E}(\sigma R)\leq c(\sigma^2 + \sigma^{-(n+2)}\mathscr{E}(R)^{q-1})\mathscr{E}(R)+c\sigma^{-(n+2)}\sqrt{\omega(R)}, 
		\end{equation} where $c=c(m,n,N,L,\ell,q)$. Set $C_{m+1}=c(m+1,n,N,L,\ell,q))$. Take $\sigma \in (0,\frac{1}{5})$ and then $\varepsilon_m \in (0,\frac{1}{3})$ such that \[ C_{m+1}\sigma^2 < \frac{1}{2}\sigma^{2\alpha}, \quad C_{m+1}\sigma^{-(n+2)}\varepsilon_m^{q-1}<\frac{1}{2}\sigma^{2\alpha}. \] In this case, with $c_1:=C_{m+1}\sigma^{-(n+2)}$, (\ref{eq:1iteramid1}) becomes \[ \mathscr{E}(\sigma R)\leq \sigma^{2\alpha}\mathscr{E}(R)+c_1\sqrt{\omega(R)}. \]
	\par To do the iteration, we consider the following
		\leqnomode \vspace{-.25cm} \begin{center} \begin{minipage}{.75\textwidth}
		\begin{align}
		&  \lvert (Du)_{B_{\sigma^jR}}\rvert \leq m+1, \label{it:1iteration1} \tag{\textbf{I}$_j$}\\
		&  \E(\sigma^jR) \leq \sigma^{2j\alpha} \E(R) +c_2\sqrt{\omega(\sigma^jR)}, \label{it:1iteration2} \tag{\textbf{II}$_j$}\\
		&  \E(\sigma^j R)\leq \varepsilon_m, \label{it:1iteration3} \tag{\textbf{III}$_j$}
		\end{align} \end{minipage} \end{center}
%		\begin{enumerate}[label=(\Roman{enumi}$_j$),leftmargin=3\parindent, itemsep=.2em]
%		\item \label{it:1iteration1} $\lvert (Du)_{B_{\sigma^jR}}\rvert \leq m+1$,
%		\item \label{it:1iteration2} $\E(\sigma^jR) \leq \sigma^{2j\alpha} \E(R) +c_2\sqrt{\omega(\sigma^jR)}$,
%		\item \label{it:1iteration3} $\E(\sigma^j R)\leq \varepsilon_m$,
%		\end{enumerate} 
where $c_2=\frac{c_1}{\sigma^{\beta}-\sigma^{2\alpha}}$. The three hold for $j=0$. Assume that they hold for $j=0,1,\dots,k-1$ with $k\geq 1$ and do induction. Then (\ref{it:1iteration3}) together with $\varepsilon_m<\frac{1}{3}$ implies $\fint_{B_{\sigma^jR}}\lvert Du-(Du)_{B_{\sigma^jR}}\rvert<1$. Combining this with (\ref{it:1iteration1}) we have, by Proposition \ref{prop:1excessest} and the choice of $\sigma, \varepsilon_m$, 
		\begin{align*}
		\E(\sigma^kR) &\leq \sigma^{2k\alpha}\E(R) +c_1 \sum_{j=0}^{k-1}\sigma^{2(k-j-1)\alpha}\sqrt{\omega(\sigma^jR)}\\
		&\leq \sigma^{2k\alpha}\E(R) +c_1 \sum_{j=0}^{k-1}\sigma^{(j-k)\beta+2(k-j-1)\alpha}\sqrt{\omega(\sigma^kR)}\\
		&\leq \sigma^{2k\alpha}\E(R) + \frac{c_1}{\sigma^{\beta}-\sigma^{2\alpha}}\sqrt{\omega(\sigma^kR)},
		\end{align*} 
which actually gives (II$_k$). Take $\sigma$ and $R_1$ small enough such that $\sigma^{2\alpha}<\frac{1}{2}, c_2\sqrt{\omega(R_1)}<\frac{\varepsilon_m}{2}$, and we furthermore have (III$_k$). Finally, to get (I$_k$) we use the triangle inequality \[ \lvert (Du)_{B_{\sigma^kR}}\rvert \leq \lvert (Du)_{B_R}\rvert + \sum_{j=0}^{k-1}\lvert (Du)_{B_{\sigma^{j+1}R}}-(Du)_{B_{\sigma^jR}}\rvert. \] For any $j\in \{0,1,\dots,k-1\}$, by Lemma \ref{lem:Emeanest2}, (\ref{it:1iteration3}) and (\ref{it:1iteration2}) we have
		\begin{align*}
		&\ \lvert (Du)_{B_{\sigma^{j+1}R}}-(Du)_{B_{\sigma^jR}}\rvert \leq \sigma^{-n}\fint_{B_{\sigma^jR}} \lvert Du-(Du)_{B_{\sigma^jR}}\rvert\\
		 \leq &\ \sigma^{-n}\sqrt{3\E(\sigma^jR)} \leq \sigma^{-n}(3\sigma^{2j\alpha}\E(R)+3c_2\sqrt{\omega(\sigma^jR)})^{\frac{1}{2}}\\
		\leq &\ \sigma^{-n}(\sqrt{3}\sigma^{j\alpha}\sqrt{\E(R)}+\sqrt{3c_2}\omega^{\frac{1}{4}}(\sigma^jR)).
		\end{align*} Sum up the above from $0$ to $k-1$ with the help of Lemma \ref{lem:omegaest1} to obtain
		\begin{align*}
		\lvert (Du)_{B_{\sigma^kR}}\rvert  &\leq m+\sqrt{3}\sigma^{-n}\sum_{j=0}^{k-1}(\sigma^{j\alpha}\sqrt{\E(R)}+\sqrt{c_2}\omega^{\frac{1}{4}}(\sigma^jR))\\
		&\leq m+\sqrt{3}\sigma^{-n}\left(\frac{\sqrt{\E(R)}}{1-\sigma^{\alpha}}+\frac{\sqrt{c_2}\beta}{2(1-\sigma^{\frac{\beta}{2}})}\Xi_{\frac{1}{4}}(R)\right).
		\end{align*} We require \[ \frac{\sqrt{3}\sigma^{-n}}{1-\sigma^{\alpha}}\sqrt{\varepsilon_m}<\frac{1}{2}, \quad \frac{\sqrt{3c_2}\beta\sigma^{-n}}{2(1-\sigma^{\frac{\beta}{2}})}\Xi_{\frac{1}{4}}(R_1)<\frac{1}{2}, \] which can be satisfied when $\varepsilon_m\ll 1, R_1 \ll 1$. Then (I$_k$) also holds true. Notice that in the above we have chosen $\sigma, \varepsilon_m,R_1$ in order such that $\lvert (Du)_{B_R}\rvert<m, \E(R)<\frac{\varepsilon_m}{2}$ and $R<R_1$ imply (\ref{it:1iteration1})-(\ref{it:1iteration3}) for any $j\in \N$. Given any $\rho\in (0,R)$, we can take $\sigma^{k+1}R <\rho\leq \sigma^kR$ and get the desired estimate for $\E(\rho)$ by controlling it with $\E(\sigma^kR)$.
%		\begin{align*}
%		\E(r)& = \fint_{B_r}E(Du-(Du)_{B_r}) \leq 4\fint_{B_r}E(Du-(Du)_{B_{\sigma^kR}})\\
%		&\leq 4\sigma^{-n}\E(\sigma^kR) \leq 4\sigma^{-n}(\sigma^{2k\alpha}\E(R)+c_2\sqrt{\omega(\sigma^kR)})\\
%		&\leq 4\sigma^{-n}(\sigma^{-2\alpha}(\frac{r}{R})^{\alpha}\E(R)+c_2\sigma^{-\beta}\sqrt{\omega(r)})\\
%		&= c_3(\frac{r}{R})^{2\alpha}\E(R)+c_4\sqrt{\omega(r)}.
%		\end{align*}
	\end{proof}
	\par We claim that there exists a relatively closed null set $S_u\subset \Omega$ such that $u \in C_{loc}^1(\Omega\setminus S_u,\R^N)$ and $Du$ locally has the modulus of continuity $\rho \mapsto \rho^{\alpha}+\Xi_{\frac{1}{4}}(\rho)$ on $\Omega \setminus S_u$. Actually, for any $x_0\in \Omega$ such that 
	\begin{equation}
	 \limsup_{R\to 0^+}\lvert (D u)_{x_0,R}\rvert<\infty \quad \mbox{and}\quad  \liminf_{R\to 0^+}\fint_{B(x_0,R)}E(D u-(D u)_{x_0,R})\id x =0,
	\end{equation} one can show that $Du =\nabla u \Le^n$ and $\nabla u$ has the desired modulus of continuity in a neighbourhood of $x_0$ by Proposition 4.9 in \cite{KriGm19} and mollifying the proof of Theorem 2.9 in \cite{Giusti03}. Thus, the set $S_u \subset \Omega$ is relatively closed and null with respect to the Lebesgue measure. 

\subsection{Improvement of regularity}\label{subsec:1improve}
	\par With the regularity proved above, it is possible to further show that the local modulus of continuity of $Du$ is $\rho\mapsto \rho^{\alpha}+\Xi_{\frac{1}{2}}(\rho)$ for any $\alpha \in(0,1)$. For any open set $\Omega^{\prime}\ssubset \Omega\setminus S_u$, we assume $\| u\|_{C^1(\Omega^{\prime})} \leq M(\Omega^{\prime})<\infty$. Then it is sufficient to perform the procedure in the quadratic case as in \cite{Giusti03}, \textsection 9.4. For completeness, we sketch the process here.
	\begin{prop}\label{prop:1caccioppoli2}
	Suppose that $F\colon \R^{N\times n} \to \R$ satisfies \textnormal{\ref{it:intp}-\ref{it:intreg}} with $p=1$, $\omega\colon [0,\infty) \to [0,\infty)$ satisfies \textnormal{\ref{it:omega1}-\ref{it:omega3}} and $u\in BV_{loc}(\Omega,\R^N)$ is an $\omega$-minimizer of $\bar{\F}$ with constant $R_0>0$. Let $S_u$ be the relatively closed singular set as in last subsection. Take $\Omega^{\prime} \ssubset \Omega\setminus S_u$ and $M=M(\Omega^{\prime})>0$ as above. For any ball $B(x_0,R)\ssubset \Omega^{\prime}$ with $R<R_0$, if $a\colon \R^n\to \R^N$ is an affine map with $\lvert \nabla a\rvert\leq m$ for some $m>0$, then there exists $C=C(m,n,N,L,\ell,M)>0$ such that 
		\begin{equation}
		\fint_{B_{\frac{R}{2}}}\lvert \nabla (u-a)\rvert^2\id x\leq C \left(\fint_{B_R}\frac{\lvert u-a\rvert^2}{R^2}\id x +\omega(R)\right).
		\end{equation}
	\end{prop}
	\begin{proof}
	In Proposition \ref{prop:1caccioppoli} we have already obtained a Caccioppoli-type inequality with respect to $E$. By (\ref{eq:Eest1}), we have $E(\frac{u-a}{R})\leq a_2 \frac{\lvert u-a\rvert^2}{R^2}$. To deal with the left-hand side, notice that $\lvert \nabla (u-a)\rvert\leq \lvert \nabla u\rvert+\lvert \nabla a\rvert\leq M+m$ and then $E(\nabla (u-a))\geq C_{M+m} \lvert \nabla (u-a)\rvert^2$.
	\end{proof}
	\begin{prop}\label{prop:1higherint}
	Suppose that $F,\omega,u, S_u, \Omega^{\prime}$ and $M(\Omega^{\prime})$ are as in Proposition \textup{\ref{prop:1caccioppoli2}} and $z_0\in \R^{N\times n}$ satisfies $\lvert z_0\rvert\leq m$. There exists $q>1$ depending on $m,n,N,L,\ell,M$ and $C=C(m,n,N,L,\ell,M,q)\geq 1$, such that $\lvert \nabla u-z_0\rvert\in L^{2q}_{loc}(\Omega^{\prime})$, and for any ball $B(y_0,R)\ssubset \Omega^{\prime}$ we have
		\begin{equation}\label{eq:1higherint}
		\left(\fint_{B(y_0,\frac{R}{2})}\lvert \nabla u-z_0\rvert^{2q}\id x \right)^{\frac{1}{q}} \leq C\left(\fint_{B(y_0,R)}\lvert \nabla u-z_0\rvert^2\id x+ \omega(R) \right).
		\end{equation}
	\end{prop}
	\begin{proof}
	Pick $B_{\rho}=B(x_0,\rho)\subset B(y_0,R)$ with $\rho <R_0$ and $a(x)=u_{B_{\rho}}+z_0(x-x_0)$. The average of $u-a$ on $B_{\rho}$ vanishes by the definition of $a$, and then the Sobolev-Poincar\'{e} inequality implies \[ \fint_{B_{\rho}}\frac{\lvert u-a\rvert^2}{\rho^2}\id x \leq C\left(\fint_{B_{\rho}}\lvert \nabla (u-a)\rvert^{2_{\ast}} \id x\right)^{\frac{2}{2_{\ast}}} =C\left(\fint_{B_{\rho}}\lvert \nabla u-z_0\rvert^{2_{\ast}}\id x \right)^{\frac{2}{2_{\ast}}}, \]where $2_{\ast}=\frac{2n}{n+2}<2$. Combining Proposition \ref{prop:1caccioppoli2} we have the weak reverse H\"{o}lder inequality
		\begin{equation}\label{eq:1wRH}
		\fint_{B_{\frac{\rho}{2}}}\lvert \nabla u-z_0\rvert^2\id x \leq C\left(\fint_{B_{\rho}}\lvert \nabla u-z_0\rvert^{2_{\ast}}\id x\right)^{\frac{2}{2_{\ast}}}+C\omega(\rho).
		\end{equation} The above estimate holds for any ball $B_{\rho} \subset B(y_0,R)$ with $\rho<R_0$ and we can replace the $\omega(\rho)$ on the right-hand side by $\omega(R)$. By the generalised Gehring lemma (see \cite{Gia83}, Chap.V or \cite{Stre84}, \textsection 2.3), we know that there is an $q_0>1$ such that $\lvert \nabla u-z_0\rvert \in L^{2q}(B(y_0,R))$ for any $q\in (1,q_0)$ with (\ref{eq:1higherint}) holding true.
	\end{proof}
	\par To get the regularity of $u$, we compare it with a harmonic map again, which is now taken as the minimizer of a quadratic functional. Take a ball $B_R=B(x_0,R)$ with $R<R_0$ and $B(x_0,2R) \ssubset \Omega^{\prime}$ and consider 
		\begin{equation}\label{eq:1ellipsystem2}
		\left\{\begin{aligned}
		&-\mbox{div}(\A\nabla h)=0,& &\mbox{in }B_R\\
		&h\vert_{\partial B_R}=u\vert_{\partial B_R},& &\mbox{on }\partial B_R,
		\end{aligned} \right. 
		\end{equation} where $\A=\tilde{F}^{\pprime}(0)$ with $\tilde{F}=F_{\nabla a}$ and $a(x)=u_{B_R}+(\nabla u)_{B_{2R}}(x-x_0)$. It is obvious that $h$ is the minimizer of \[ \mathscr{G}(v,B_R) := \int_{B_R}(F(\nabla a)+F^{\prime}(\nabla a) \nabla(v-a)+\frac{1}{2}F^{\pprime}(\nabla a)[\nabla(v-a),\nabla(v-a)])\id x. \]
	\begin{lem}\label{lem:1harapprox2}
	Let $F,\omega,u, S_u,\Omega^{\prime}, M(\Omega^{\prime})$ be as in Proposition \textnormal{\ref{prop:1caccioppoli2}}, $\tilde{F},a, B_R$ and $h$ be as above, and $q$ be the exponent obtained in Proposition \textnormal{\ref{prop:1higherint}}. Then for some $C=C(n,N,L,\ell,M,q)>0$ we have
		\begin{equation}\label{eq:1harapprox2}
		\fint_{B_R}\lvert \nabla (u-h)\rvert^2\id x \leq C\left(\fint_{B_{2R}}\lvert \nabla(u-a)\rvert^2 \id x \right)^{1+\frac{1}{q^{\prime}}}+ C\omega(2R).
		\end{equation}
	\end{lem}
	\begin{proof}
	As $u\in C^1(\bar{B}_R,\R^N)$, we have $\lvert \nabla (u-h)\rvert\in L^2(B_R)$ and by (\ref{eq:LH}) 
		\begin{align*}
		\int_{B_R} \lvert \nabla (u-h)\rvert^2\id x &\leq \frac{1}{2}C\int_{B_R}\tilde{F}^{\pprime}(0)[\nabla(u-h),\nabla (u-h)]\id x\\
		& = C (\mathscr{G}(u)-\mathscr{G}(h))\\
		&=C(\mathscr{G}(u)-\mathscr{F}(u)+\mathscr{F}(u)-\mathscr{F}(h)+\mathscr{F}(h)-\mathscr{G}(h))\\
		&=: C(I+II+III).
		\end{align*} The $\omega$-minimality of $u$, H\"{o}lder's inequality and the $L^2$-estimate of (\ref{eq:1ellipsystem2}) (see \cite{Giusti03}, \textsection 10.4) give \[ II \leq \omega(R)\int_{B_R}(1+\lvert \nabla h\rvert)\id x \leq \omega(R)(\omega_nR^n+CR^{\frac{n}{2}}\| \nabla u\|_{L^2(B_R)}) \leq \omega(R)\omega_nR^n(1+CM). \] By the $C^1$ boundedness of $u$, we have $\lvert \nabla (u-a)\rvert\leq 2M$ and then $I$ can be estimated as follows
		\begin{align*}
		I &=-\int_{B_R} \int_0^1(1-t)(F^{\pprime}(\nabla a+t\nabla (u-a))-F^{\pprime}(\nabla a))[\nabla(u-a),\nabla (u-a)]\id x\\
		&\leq C(M)\int_{B_R}\lvert \nabla(u-a)\rvert^3 \id x\\
		&\leq C(M)\omega_nR^n\left(\fint_{B_R} \lvert \nabla (u-a)\rvert^{q^{\prime}}\id x \right)^{\frac{1}{q^{\prime}}} \left(\fint_{B_R} \lvert \nabla (u-a)\rvert^{2q}\id x \right)^{\frac{1}{q}}\\
		&\overset{(\ref{eq:1higherint})}{\leq} C\omega_nR^n \left(\fint_{B_R} \lvert \nabla (u-a)\rvert^2 \id x \right)^{\frac{1}{q^{\prime}}}\left(\fint_{B_{2R}}\lvert \nabla (u-a)\rvert^2\id x +\omega(2R) \right).
		\end{align*} The $q$ can be taken smaller than $2$ and thus $\lvert \nabla (u-a)\rvert^{q^{\prime}} \leq (2M)^{q^{\prime}-2}\lvert \nabla (u-a)\rvert^2$. The estimate of $III$ is similar with the help of the $L^p$-estimate of (\ref{eq:1ellipsystem2}) (see \cite{Giusti03}, Section 10.4). Summing up the estimates for $I, II$ and $III$ gives the desired inequality.
	\end{proof}
	\begin{prop}\label{prop:excessest2}
	Suppose that $F,\omega,u, S_u, \Omega^{\prime}, M(\Omega^{\prime})$ and $q$ are as in Proposition \textnormal{\ref{prop:1higherint}}. Take a ball $B(x_0,R)$ such that $R<R_0$ and $B(x_0,2R) \ssubset \Omega^{\prime}$. For any $\sigma,\gamma \in (0,1)$ we have 
		\begin{equation}\label{eq:excessest2}
		\E_1(\sigma R) \leq C(\sigma^{-n}+\sigma^{2\gamma})\left(\E_1(2R)^{1+\frac{1}{q^{\prime}}}+\omega(2R) \right)+C\sigma^{2\gamma}\E_1(2R)
		\end{equation} for any $\gamma \in (0,1)$ with $C=C(n,N,L,\ell,M,q,\gamma)>0$, where $\E_1$ is the $L^2$-excess \[ \E_1(x_0,\rho) := \fint_{B_{\rho}}\lvert \nabla u-(\nabla u)_{B_{\rho}}\rvert^2 \id x. \]
	\end{prop}
	\begin{proof}
	Suppose that $h$ is as in Lemma \textup{\ref{lem:1harapprox2}}. For any $\rho<R$, the harmonic function $h$ satisfies, by \textsection III.2 in \cite{Gia83},
		\begin{equation}
		\fint_{B_{\rho}} \lvert \nabla h-(\nabla h)_{B_{\rho}}\rvert ^2 \id x \leq C\left(\frac{\rho}{R}\right)^{2\gamma} \fint_{B_R} \lvert \nabla h-(\nabla h)_{B_R}\rvert^2 \id x.
		\end{equation} Then the excess of $\nabla u$ can be estimated by comparing $\nabla u$ and $\nabla h$. With the help of (\ref{eq:1harapprox2}), we can obtain (\ref{eq:excessest2}).
%		\begin{align*}
%		\fint_{B_{\rho}}\lvert \nabla u-(\nabla u)_{B_{\rho}}\rvert &\leq 6\fint_{B_{\rho}}\lvert \nabla u-\nabla h\rvert^2\id x +3\fint_{\B_{\rho}} \lvert \nabla h-(\nabla h)_{B_{\rho}}\rvert^2 \id x\\
%		&\leq 6(\frac{R}{\rho})^n \fint_{B_R}\lvert \nabla u-\nabla h\rvert^2\id x + 3C (\frac{\rho}{R})^{2\gamma}\fint_{B_R}\lvert \nabla h-(\nabla h)_{B_R}\rvert^2 \id x,
%		\end{align*} where \[ \fint_{B_R}\lvert \nabla h-(\nabla h)_{B_R}\rvert^2 \id x\leq 6\fint_{B_R}\lvert \nabla u-\nabla h\rvert^2\id x +3\fint_{B_R}\lvert \nabla u-(\nabla u)_{B_R}\rvert^2\id x. \] Taking Lemma \ref{lem:1harapprox2} into consideration we have the following desired estimate
%		\begin{align*}
%		&\ \fint_{B_{\rho}}\lvert \nabla u-(\nabla u)_{B_{\rho}}\rvert^2 \id x \\
%		\leq&\ C\left((\frac{R}{\rho})^n+(\frac{\rho}{R})^{2\gamma} \right)\fint_{B_R}\lvert \nabla (u-h)\rvert^2\id x +C(\frac{\rho}{R})^{2\gamma} \fint_{B_R} \lvert \nabla u-(\nabla u)_{B_R}\rvert^2 \id x\\
%		\leq&\ C\left((\frac{R}{\rho})^n+(\frac{\rho}{R})^{2\gamma} \right)\left( \left(\fint_{B_{2R}}\lvert \nabla u-(\nabla u)_{B_{2R}}\rvert^2 \id x\right)^{1+\frac{1}{q^{\prime}}}+\omega(2R) \right)\\
%		&\quad \quad +C(\frac{\rho}{R})^{2\gamma} \fint_{B_{2R}} \lvert \nabla u-(\nabla u)_{B_{2R}}\rvert^2 \id x.
%		\end{align*}
	\end{proof}
	\par Replace $2R$ by $R$ and then the excess estimate is
		\begin{equation}
		\E_1(\sigma R) \leq C(\sigma^{-n}+\sigma^{2\gamma})\left(\E_1(R)^{1+\frac{1}{q^{\prime}}}+\omega(R) \right)+C\sigma^{2\gamma}\E_1(R).
		\end{equation} It indeed holds for $\sigma \in(0,1)$ as the case $\sigma\in (\frac{1}{2},1)$ is obvious. Given $\alpha \in (0,1)$, we take $\gamma >\alpha$ and do iteration as in Subsection \ref{subsec:1iteration}. The final statement is as follows: There exist $\varepsilon_0>0$ and $R_2\in(0,R_0)$ such that if $B(x_0,R)\ssubset \Omega^{\prime}$ satisfies \[ \E_1(R)<\frac{\varepsilon_0}{2},\quad R<R_2, \] we have 
		\begin{equation}\label{eq:schauder2}
		\E_1(\rho) \leq C\left(\frac{\rho}{R}\right)^{2\alpha}\E_1(R) +C\omega(\rho)
		\end{equation} for any $\rho\in(0,R)$ with some $C=C(n,N,L,\ell,M, \alpha)>0$. It is routine to get that $\nabla u$ has the local modulus of continuity $\rho \mapsto \rho^{\alpha}+\Xi_{\frac{1}{2}}(\rho)$. From the discussion at the end of last subsection, it is not hard to see that $S_u \subset \Sigma_1 \cup \Sigma_2$, where
		\begin{align*}
		\Sigma_1 &:= \left\{x\in \Omega: \liminf_{\rho \to 0^+}\fint_{B_{\rho}(x)}E(Du-(Du)_{B_{\rho}(x)}) >0\right\},\\
		\Sigma_2 &:= \left\{x\in \Omega: \limsup_{\rho\to 0^+}\lvert (Du)_{B_{\rho}(x)}\rvert=\infty\right\}.
		\end{align*}
		
\subsection{Indirect argument}\label{subsec:1ind}	
	\par In \cite{DGK05} the authors showed the partial regularity for $\omega$-minimizers in the subquadratic case ($1<p<2$), where the harmonic approximation is done via an indirect argument. That method can be adapted to the linear growth case in this paper, and a sketch of the proof is in the following. We remark that with this method, only $C^2$ regularity of $F$ is needed, in other words, we replace \ref{it:intreg} by
		\begin{enumerate}[label=(\textbf{H3})$^{\prime}$, leftmargin=3\parindent, itemsep=.2em]
		\item\label{it:int2decont} $F$ is $C^2$ and \[\abs{F^{\pprime}(z_1)-F^{\pprime}(z_2)} \leq \nu_M(\abs{z_1-z_2})\] for any $z_1,z_2 \in B(0,M+1)$, where $\nu_M$ is concave and non-decreasing on $[0,\infty)$ with $\nu_M(0)=\lim_{t\to 0}\nu_M(t)=0$.
		\end{enumerate}	The Dini type condition of $\omega$ can also be relaxed to $\Xi_{\frac{1}{2}}(\rho)<\infty$ for any $\rho>0$, as the desired exponent of $\omega(R)$ is obtained with one attempt in the excess decay estimate.
	\par For this argument, most of the steps in \cite{DGK05} remain the same. The difference is twofold: the Sobolev-Poincar\'{e} inequality and the harmonic approximation.
	\par Define two maps $V, W$ on finite dimensional Hilbert spaces (not specified here): \[ V(\xi) := \frac{\xi}{(1+\lvert \xi\rvert^2)^{\frac{1}{4}}}, \quad W(\xi):=\frac{\xi}{\sqrt{1+\lvert \xi\rvert}}. \] Then we can see that $\lvert W(\xi)\rvert \leq \lvert V(\xi)\rvert\leq 2^{\frac{1}{4}}\lvert W(\xi)\rvert$ and $\lvert W(\cdot)\rvert^2$ is convex.
	\begin{thm}\label{thm:SPW}
	Let $B_{R} = B(x_0,R)\subset \R^n$ be a ball with $n\geq 2$. Then for any $u\in BV(B_{R},\R^N)$ there holds
		\begin{equation}
		\left(\fint_{B_{R}}\left\lvert W\left(\frac{u-u_{R}}{R}\right) \right\rvert^{\frac{2n}{n-1}}\id x \right)^{\frac{n-1}{2n}} \leq c_s \left(\fint_{B_{R}}\lvert W(Du)\rvert^2 \right)^{\frac{1}{2}},
		\end{equation} where the constant $c_s$ depends on $n,N$. It also holds with $W$ replaced by $V$.	 
	\end{thm}
	\begin{proof}
	Notice that $\lvert W\rvert^2$ is convex and thus $\int \lvert W\rvert^2$ is continuous with respect to convergence in the area-strict sense in $BV$. Thus, we only need to consider maps in $W^{1,1}\cap C^{\infty}(B_{R},\R^N)$, and the general case follows by approximation. For $x,y \in B_{R}$, it is easy to see $\lvert x-y\rvert<2R$. Then fix an $x\in B_{R}$, by Theorem 2 in \cite{DGK05} we have \[ W^2\left( \frac{\lvert u-u_{R}\rvert}{2R}\right) \leq \frac{(2R)^{n-1}}{(n-1)\Le^n(B_{R})} \int_{B_{R}} \frac{W^2(\lvert Du(y)\rvert)}{\lvert x-y\rvert^{n-1}} \id y. \]
%	\[ \frac{\lvert u(x)-u_{x_0,R}\rvert}{2R} \leq \fint_{B_{R}(x_0)}\frac{\lvert u(x)-u(y)\rvert}{2R} \id y \leq \fint_{B_{R}(x_0)} \fint_0^{\lvert x-y\rvert}\lvert Du(x+r\omega)\omega\rvert\id r \id y, \] where $\omega = \frac{y-x}{\lvert y-x\rvert}$. The function $W^2(t)$ is non-decreasing in $t\in [0,\infty)$ and convex, then we can apply it to both sides and use Jensen's inequality to obtain \[ W^2\left( \frac{\lvert u(x)-u_{x_0,R}\rvert}{2R}\right) \leq \fint_{B_{R}(x_0)} \fint_0^{\lvert x-y\rvert}W^2(\lvert Du(x+r\omega)\omega\rvert)\id r \id y. \] Define $\widetilde{W}(\lvert Du(z)\rvert) = W(\lvert Du(z)\rvert) \chi_{B_{R}(x_0)}(z)$ and consider the balls centered at $x$. Then the right-hand side can be estimated as follows:
%		\begin{align*}
%		\mbox{RHS} &\leq \frac{1}{\Le^n(B_{R}(x_0))}\int_{B_{2R}(x)} \fint_0^{\lvert x-y\rvert} \widetilde{W}^2(\lvert Du(x+r\omega)\rvert)\id r\id y \\
%		&\leq \frac{1}{\Le^n(B_{R}(x_0))}\int_{\SP^{n-1}}\int_0^{2R} s^{n-2}\int_0^s \widetilde{W}^2(\lvert Du(x+r\omega)\rvert)\id r \id s \id \omega \\
%		& \leq \frac{1}{\Le^n(B_{R}(x_0))}\int_{\SP^{n-1}} \int_0^{2R} \widetilde{W}^2(\lvert Du(x+r\omega)\rvert) \int_s^{2R} s^{n-2}\id s \id r \id \omega\\
%		& \leq \frac{(2R)^{n-1}}{(n-1)\Le^n(B_{R}(x_0))} \int_{B_{R}(x_0)} \frac{W^2(\lvert Du(y)\rvert)}{\lvert y-x\rvert^{n-1}} \id y.
%		\end{align*} 
Integrating with respect to $x$ in $B_{R}$, we get 
		\begin{align}
		&\ \fint_{B_{R}}\left\lvert W\left(\frac{u-u_{R}}{R}\right) \right\rvert^{2}\id x \leq \frac{c}{R} \fint_{B_{R}} \int_{B_{R}} \frac{W^2(\lvert Du(y)\rvert)}{\lvert x-y\rvert^{n-1}}\id y\id x \label{eq:WSPpoincare} \\
		&\leq \frac{c}{R} \fint_{B_{R}}W^2(\lvert Du(y)\rvert)\int_{B_{2R}(y)} \lvert x-y\rvert^{1-n}\id x \id y \leq c\fint_{B_{R}} W^2(\lvert Du(y)\rvert) \id y. \notag
		\end{align} To get a higher order integrability of $W(\lvert u-u_{R}\rvert/R)$, we need the classical Sobolev inequality. Consider $g = (u-u_{R})/R$ and $U = W^2(\lvert g\rvert)$. Notice that $W^2(\lvert \cdot\rvert)$ is Lipschitz, and then $U \in W^{1,1}(B_{R})$ with \[ DU(x) =  \frac{\lvert g(x)\rvert(2+\lvert g(x)\rvert)}{(1+\lvert g(x)\rvert)^2}Dg(x) \frac{g(x)}{\lvert g(x)\rvert} \quad \mbox{ in } \{x\in B_R: g(x)\neq 0 \}. \] The Sobolev embedding for $W^{1,1}$ gives 
		\begin{equation}\label{eq:WSPmid}
		 \left(\fint_{B_{R}}\lvert U\rvert^{\frac{n}{n-1}} \id x\right)^{\frac{n-1}{n}} \leq C\left( R \fint_{B_{R}} \lvert DU\rvert\id x  + \fint_{B_{R}}\lvert U\rvert\id x \right). 
		\end{equation} When $\lvert Du(x)\rvert\geq 1$, from the expression of $DU$ we have \[ R \lvert DU(x)\rvert\leq 2 \lvert Du(x)\rvert \leq cW^2(\lvert Du(x)\rvert). \] When $0<\lvert Du(x)\rvert<1$, apply Young's inequality and then
		\begin{align*}
		R \lvert DU(x)\rvert& \leq \frac{1}{2}\frac{\lvert g\rvert^2(2+\lvert g\rvert)^2}{(1+\lvert g\rvert)^4} + \frac{1}{2}\lvert Du\rvert^2 \\
		& \leq 2\min\{\lvert g\rvert,\lvert g\rvert^2\} + c W^2(\lvert Du\rvert)\\
		& \leq c(W^2(\lvert g\rvert)+W^2(\lvert Du\rvert)).
		\end{align*} Thus, the first term on the right-hand side of (\ref{eq:WSPmid}) is controlled by \[ R\fint_{B_{R}}\lvert DU\rvert\id x \leq C\fint_{B_{R}} \left(W^2\left(\frac{u-u_{R}}{R}\right) + W^2(\lvert Du\rvert)\right)\id x. \] Combining (\ref{eq:WSPpoincare}) we have the desired inequality.
	\end{proof}
	\par We have obtained a Caccioppoli-type inequality in Proposition \ref{prop:1caccioppoli}. It is easy to see that $V^2(t) \sim E(t)$, so we have the Caccioppoli-type inequality with respect to $V^2$. 
	\begin{lem} \label{lem:indcaccioppoli}
	Suppose that $F\colon \R^{N\times n}\to \R$ satisfies \textnormal{\ref{it:intp},\ref{it:intqc}} and \textnormal{\ref{it:int2decont}} with $p=1$, $\omega$ satisfies \textnormal{\ref{it:omega1}} with constant $R_0>0$, and $u \in BV_{loc}(\Omega,\R^N)$ is an $\omega$-minimizer of $\bar{\F}$. Fix $m>0$, then there exists $c_c=c_c(m,n,N,L,\ell)$ such that for any $B_{R}=B(x_0,R)\ssubset \Omega$ with $R <R_0$ and affine map $a: \R^n\to \R^N$ with $\lvert \nabla a\rvert\leq m$, there holds 
		\begin{equation}
		\fint_{B_{\frac{R}{2}}} \lvert V(D(u- a))\rvert^2 \leq c_c \left(\fint_{B_{R}}\left\lvert V\left(\frac{u-a}{R} \right)\right\rvert^2 \id x + \omega(R)\right).
		\end{equation}
	\end{lem}	
	\par The $\omega$-minimality of $u$ implies that it is almost an $\A$-harmonic map with a proper $\A$, to present which we define the excess for $u$ with $A \in \R^{N\times n}$: \[\E_2(x_0,R,A) := \left(\fint_{B(x_0,R)} \lvert V(Du)-V(A)\rvert^2 \right)^\frac{1}{2}.  \] When $x_0$ (and $R$) and $A$ are fixed, we abbreviate the quantity as $\E_2(R) (\E_2)$. The following can be showed with the proof of Lemma 4 in \cite{DGK05} by considering $\nabla u$ and $D^su$ separately.
	\begin{lem}[Approximate harmonicity]\label{lem:indalmosthar}
	Suppose that $F, \omega$ and $u$ are as in Lemma \textnormal{\ref{lem:indcaccioppoli}}. For any $m>0$, there exists $c_e>0$ depending on $mn,N,L$ such that for any ball $B_{R}=B(x_0,R)\ssubset \Omega$ with $R<R_0$ and any $A \in \R^{N\times n}$ with $\lvert A\rvert\leq m$, we have 
		\begin{equation}
		\left\lvert \fint_{B_{R}} F^{\pprime}(A)[Du-A,D\varphi] \right\rvert \leq c_e(\sqrt{\nu_M(\E_2)}\E_2+\E_2^2 +\sqrt{\omega(R)})\sup_{B_{R}}\lvert D\varphi\rvert
		\end{equation} for any $\varphi \in C_0^1(B_{R},\R^N)$.
	\end{lem}
	With this result, we are able to approximate $u$ by an $\A$-harmonic map by the following lemma:
	\begin{lem}\label{lem:indharapprox}
	For any $\varepsilon>0$, there exists $\delta = \delta(n,N,\Lambda,\lambda,\varepsilon) \in (0,1]$ such that for any $\A \in \bigodot^2(\R^{N\times n})$ that satisfies \textnormal{(\ref{eq:LHcondition})}, any ball $B_R=B(x_0,R)\subset \R^n$ and any $v\in BV(B_{R},\R^n)$ with 
		\begin{align}
		&\fint_{B_{R}}\lvert W(Dv)\rvert^2 \leq \gamma^2 \leq 1,\\
		&\fint_{B_{R}}\A[Dv,D\varphi] \leq \gamma \delta \sup_{B_{R}}\lvert D\varphi\rvert, \quad \mbox{for any }\varphi \in C^1_0(B_{R},\R^N),
		\end{align} there exists an $\A$-harmonic map $h$ satisfying
		\begin{equation}
		\fint_{B_{R}}\lvert W(Dh)\rvert^2 \leq 1, \quad \fint_{B_{R}}\left\lvert W\left(\frac{v-\gamma h}{R}\right)\right\rvert^2\id x \leq \gamma^2 \varepsilon.
		\end{equation}
	\end{lem}	
	\par The proof of this lemma is by contradiction, see Lemma 6 in \cite{DGK05}. The integrals concerning $\nabla u$ and $D^su$ need to be considered separately when necessary. Notice that the scaling between $B_{R}$ and $B(0,1)$ for $BV$ maps does not hold straightforward but can proved by approximation with $W^{1,1}\cap C^{\infty}$ maps.
	\par The excess decay estimate can be done with the same procedure as that in \cite{DGK05},  Lemma 7. At some points we need to consider the singular part of the integral of a $BV$ map separately, which will not make an essential difference.
%	Consider $x_0\in \Omega$ such that $\lvert (Du)_{x_0,R}\rvert \leq m$. Let $A =(Du)_{x_0,R}$ and $\A = D^2F(A)$. For $\sigma \in (0,\frac{1}{4}]$, we take $\varepsilon = \sigma^{n+4}$. Then from Lemma \ref{lem:indharapprox} we have $\delta = \delta(n,N,\ell,L,\sigma) \in (0,1]$. The constants $c_e = c_e(n,N,L,m),c_a=c_a(n,N,\frac{\Lambda}{\lambda})$ are as in Lemma \ref{lem:indalmosthar} and Lemma \ref{lem:ellip1}, where $\lambda,\Lambda$ are determined by $\ell$ and $L$. Set $\E_2(\rho) = \E_2(x_0,\rho,(Du)_{x_0,\rho})$ and 
%		\begin{equation}
%		\Gamma(R)=\sqrt{\E_2^2(R)+4\frac{\omega(R)}{\delta^2}},\quad v = u-(Du)_{R}(x-x_0),\quad \gamma = c_1c_e\Gamma(R),
%		\end{equation} where $c_1$ is such that (see Lemma 1 (vi) in \cite{DK02}) \[ \lvert V(\xi)-V(\eta)\rvert\leq c_1\lvert V(\xi-\eta)\rvert, \quad \mbox{for any }\lvert \eta\rvert\leq m. \] Then we can apply Lemma \ref{lem:indalmosthar} to $u$ and Lemma \ref{lem:indharapprox} to $v$ to get:
%	\begin{lem}\label{lem:indexcessdecay}
%	Assume that for some $R \in (0,R_0]$ the following holds:
%		\begin{equation}
%		\lvert (Du)_{x_0,R}\rvert\leq m,\  2\sqrt{2}c_a\gamma \leq 1,\  \sqrt{\nu_M(\E_2(R))} +\E_2(R) \leq \frac{\delta}{2},\  c_ec_a \E_2(R) \leq 1.
%		\end{equation} Then there exist constants $\tilde{c} = \tilde{c}(n,N,L,\ell, m)$ and $\hat{c}=\hat{c}(n,N,L,\ell,m,\sigma)$ such that 
%		\begin{equation}
%		\E_2^2(\sigma R)\leq \tilde{c}\sigma^2 \E_2^2(R)+\hat{c}\omega(R).
%		\end{equation}
%	\end{lem} 
	\par Then if a ball $B_R=B(x_0,R)\ssubset \Omega$ is such that $\lvert (Du)_R\rvert\leq m$, and $\E_2(R)$ and $R$ are taken to be small enough, we will have \[ \E_2^2(\rho) \leq C\left(\left(\frac{\rho}{R}\right)^{2\alpha}\E_2(R)+\omega(R)\right), \quad \mbox{for any }\rho\in (0,R). \]  The desired partial regularity hence follows.
		
\section{Partial regularity for $u$}\label{sec:upartialreg}
	 \par This section is for the proof of Theorem \ref{thm:upartialreg}, which gives the partial H\"{o}lder regularity of $\omega$-minimizers in the subquadratic case without the Dini-type condition \ref{it:omega3}. The main steps are similar with those in last section, so we omit some details and only present an outline with the difference.
	
\subsection{Caccioppoli-type inequality}\label{subsec:caccio2}
	\par To show Theorem \ref{thm:upartialreg}, we need to consider a normalised excess (see (\ref{eq:uexcess})). Correspondingly, the Caccioppoli-type inequality in this case also contains a normalising factor $(1+\lvert A\rvert)$.
	\begin{prop}\label{prop:caccio2}
	Suppose that $F\colon \R^{N\times n}\to \R$ satisfies \textnormal{\ref{it:intp}, \ref{it:intqc}, \ref{it:int2deriv}} and \textnormal{\ref{it:int2deLip}} with $p\in(1,2)$, and $\omega$ satisfies \textnormal{\ref{it:omega1}}. The map $u\in BV_{loc}(\Omega,\R^N)$ is an $\omega$-minimizer of $\F$ with constant $R_0>0$. Then for any ball $B_R=B(x_0,R) \ssubset \Omega$ with $R<R_0$ and any affine map $a\colon \R^n\to \R^N$ with $\nabla a =A\in R^{N\times n}$, there exists a constant $c=c(n,N,L,\ell,p)$ independent of $a$ such that 
		\begin{equation}\label{eq:ucaccio}
		\int_{B_{\frac{R}{2}}}E_p\left(\frac{\nabla u-A}{1+\lvert A\rvert}\right)\id x \leq c\left(\int_{B_R}E_p\left(\frac{u-a}{R(1+\lvert A\rvert)} \right)\id x + \omega(R)R^n\right).
		\end{equation}
	\end{prop}
	\begin{proof}
	Set $\tilde{F}:= F_A$, $\tilde{u}=u-a$, and fix $\frac{R}{2}<t<s<R$. Take a smooth cut-off function between $B_t$ and $B_s$ with $\rho \in C_c^{\infty}(B_s)$ and $\lvert \nabla \rho\rvert\leq \frac{2}{s-t}$, and set $\varphi = \rho\tilde{u}, \psi = (1-\rho)\tilde{u}$. Then $\varphi \in W^{1,p}_0(B_s,\R^N)$, and the quasiconvex condition \ref{it:intqc} with (\ref{eq:pshiftedqc}) gives \[\int_{B_s}\tilde{F}(\nabla \varphi_{\varepsilon})\id x \geq c\,\ell \int_{B_s}E_p^A(\nabla \varphi)\id x.\] The rest part can be carried out as in Proposition \ref{prop:1caccioppoli} with $E$ replaced by $E_p^A$. We estimate the term with $\omega(s)$ as follows
		\begin{align*}
		 \omega(s)\int_{B_s} (1+\abs{\nabla \psi}^p)\id x &= \omega(s)(1+\abs{A})^p\int_{B_s}\frac{1+\abs{\nabla \psi}^p}{(1+\abs{A})^p}\id x\\
		&\overset{(\ref{eq:Eest3})}{\leq} \omega(s)(1+\abs{A})^p\int_{B_s}\brac{2+\frac{1}{a_1}E_p\brac{\frac{\nabla \psi}{1+\abs{A}}}}\id x\\
		&\leq 2\omega(R) \omega_nR^n(1+\abs{A})^p+ \frac{1}{a_1}\int_{B_s}E_p^A(\nabla \psi)\id x.
		\end{align*}	The inequality obtained from above with Lemma \ref{lem:iteraineq} is
		\begin{equation}
		\int_{B_{\frac{R}{2}}} E_p^A(\nabla u-A) \leq C \int_{B_R}E_p^A\brac{\frac{u-a}{R}}\id x + C\omega(R)R^n(1+\abs{A})^p,
		\end{equation} and then (\ref{eq:ucaccio}) follows by (\ref{eq:Ehomog}).
	\end{proof}
		
\subsection{Harmonic approximation}
	\par The result in this subsection can be obtained by modifying the process in Subsection \ref{subsec:ELineq} and \ref{subsec:harapprox}, so we will omit the repetitive part and only give the difference.	
	\par Suppose that $F,\omega$ and $u$ are as in Theorem \ref{thm:upartialreg}, where $p\in(1,2)$. Take $B_R=B(x_0,R)\ssubset \Omega$ with $R<R_0$, and fix $A \in\R^{N\times n}$. Similar with Subsection \ref{subsec:ELineq}, we have
		\begin{equation}
		\F(u,B_R) \leq \inf_{v \in W^{1,1}_u(B_R,\R^N)}\mathscr{F}(v,B_R)+\omega_nR^n \varepsilon,
		\end{equation} where $\varepsilon = \omega(R)\fint_{B_R}(a_6+a_7\lvert \nabla u\rvert^p)\id x$. Consider the complete metric space $X = W^{1,p}_u(B_R,\R^N)$ with \[d(w_1,w_2) = (1+\abs{A})^{\frac{p}{2}-1}\brac{\fint_{B_R}\abs{\nabla (w_1-w_2)}^p\id x}^{\frac{1}{p}}.\] The Ekeland variational principle (Lemma \ref{lem:Ekeland}) then implies the existence of $w \in W^{1,p}_u(B_R,\R^N)$ such that, with $\mathcal{F}(u) = \fint_{B_R}F(\nabla u)\id x$,
		\begin{enumerate}[label=(\alph*)]
		\item\label{it:uEkeland1} $d(u,w) \leq \sqrt{\varepsilon}$;
		\item $\mathcal{F}(w) \leq \mathcal{F}(u)$;
		\item\label{it:uEkeland3} $\mathcal{F}(w) \leq \mathcal{F}(v) + \sqrt{\varepsilon}\,d(w,v)$, for any $v \in X=W^{1,p}_u(B_R,\R^N)$.
		\end{enumerate} Subsequently, we have the Euler-Lagrange inequality: for any $\varphi \in W^{1,p}_0(B_R,\R^N)$ there holds
		\begin{equation}\label{eq:uELineq}
		\abs{\fint_{B_R}F^{\prime}(\nabla w)\cdot \nabla \varphi \id x } \leq \sqrt{\varepsilon}(1+\abs{A})^{\frac{p}{2}-1} \brac{\fint_{B_R}\abs{\nabla \varphi}^p\id x}^{\frac{1}{p}}.
		\end{equation}
	\begin{prop}\label{prop:uharapprox}
	Suppose that $F\colon\, \RNn\to \R$ satisfies \textnormal{\ref{it:intp}, \ref{it:intqc}, \ref{it:int2deriv}} and \textnormal{\ref{it:int2deLip}} with $p\in(1,2)$, and $\omega$ satisfies \textnormal{\ref{it:omega1}}. The map $u\in W^{1,p}_{loc}(\Omega,\R^N)$ is an $\omega$-minimizer of $\F$ with constant $R_0>0$. For any ball $B_R=B(x_0,R) \ssubset \Omega$ and any affine map $a\colon \R^n\to \R^N$ with $\nabla a = A\in \R^{N\times n}$, the system
		\begin{equation}\label{eq:uellipsystem}
		\left\{\begin{aligned}
		&-\mathrm{div}(F^{\pprime}(A)\nabla h) = 0,& &\mathrm{in }\ B_R\\
		&h\vert_{\partial B_R}=u\vert_{\partial B_R},& &\mathrm{on }\ \partial B_R
		\end{aligned} \right.
		\end{equation} admits a unique solution $h \in W_u^{1,p}(B_R,\R^N)$ such that 
		\begin{equation}\label{eq:uharapproxest}
		\left(\fint_{B_R}\lvert \nabla h-A\rvert^p\id x \right)^{\frac{1}{p}} \leq C \brac{\fint_{B_R}\abs{\nabla u-A}^p\id x}^{\frac{1}{p}},
		\end{equation} where $C=C(n,N,\frac{L}{\ell},p)>0$. Furthermore, set \[\varepsilon = \fint_{B_R}(a_6+a_7\abs{\nabla u}^p)\id x,\ \varepsilon_{A,p} = \frac{\varepsilon}{(1+\abs{A})^p},\ r= \max\left\{2,\frac{np^{\prime}}{n+p^{\prime}}\right\}, \] and denote $\frac{r^{\prime}}{p} = \min\{\frac{2}{p},\frac{n}{n-p}\}$ by $s$, then there exists a constant $C=C(n,N,L,\ell,p)>0$ such that
		\begin{equation}\label{eq:uharapprox}
		\fint_{B_R} E_p\left(\frac{u-h}{R(1+\lvert A\rvert)}\right)\id x \leq  C \left(\fint_{B_R}E_p\left(\frac{\nabla u-A}{1+\lvert A\rvert} \right) \id x\right)^s + C(\varepsilon_{A,p}^{\frac{p}{2}}+\varepsilon_{A,p}^{\frac{r^{\prime}}{2}}) .
		\end{equation}
	\end{prop}
	\begin{proof}
	Define $\A := F^{\pprime}(A)(1+\abs{A})^{2-p}$. Then from \ref{it:int2deriv} and Lemma \ref{lem:LH} we know that $\abs{\A}\leq L$ and the operator satisfies the Legendre-Hadamard condition. Lemma \ref{lem:ellip1} and the comment after it indicate that there exists a unique solution $h\in W^{1,p}_u(B_R,\R^N)$ to (\ref{eq:uellipsystem}) satisfying (\ref{eq:uharapproxest}).
	\par Set $\tilde{F}=F_A, \tilde{u}=u-a$ and $\tilde{w}=w-a$. As in (\ref{eq:1harmid1}), we have, by Lemma \ref{lem:shiftedest3}, (\ref{eq:uELineq}), H\"{o}ler's inequality and the fact $E(z)^p \leq c(p)E_p(z)$,
		\begin{align}
		&\ \fint_{B_R} \tilde{F}^{\pprime}(0)[\nabla(w-h),\nabla \varphi]\id x \label{eq:uharmid1}\\
		\leq&\ (1+\abs{A})^{p-1}\brac{C\fint_{B_R}E\brac{\frac{\nabla \tilde{w}}{1+\lvert A\rvert}}\abs{\nabla \varphi}\id x  + \varepsilon^{\frac{1}{2}}_{A,p}\fint_{B_R}\lvert \nabla \varphi\rvert\id x } \notag\\
		\leq&\ (1+\abs{A})^{p-1} \brac{\fint_{B_R}\abs{\nabla \varphi}^{p^{\prime}} \id x}^{\frac{1}{p^{\prime}}}\brac{C\brac{\fint_{B_R}E_p\brac{\frac{\nabla \tilde{w}}{1+\abs{A}}}\id x}^{\frac{1}{p}} +\varepsilon^{\frac{1}{2}}_{A,p}} \notag
		\end{align} for any $\varphi \in W^{1,\infty}_0\cap C^1(B_R,\R^n)$. To find a proper test map $\varphi$, we again scale to the unit ball $B=B(0,1)$, define $\Phi, \Psi$ and $\tilde{W}$ as Proposition \ref{prop:1harapprox} and consider 
		\begin{equation}\label{eq:udualelliptic}
		\left\{\begin{aligned}
		&-\mbox{div}(\mathbb{A}\nabla \Phi) = T_p\brac{\frac{\Psi}{1+\lvert A\rvert}},& &\mbox{in }B\\
		&\Phi\vert_{\partial B} = 0,& &\mbox{on }\partial B,
		\end{aligned} \right.
		\end{equation} where for any $y\in \R^N$ \[ T_p(y)= \left\{
		\begin{aligned}
		&y,& &\abs{y}\leq 1\\
		&\abs{y}^{p-2}y,& &\abs{y}>1.
		\end{aligned}\right. \] Then we have $T_p(\frac{\Psi}{1+\abs{A}})\in L^{p^{\prime}}(B,\R^N)$ and that (\ref{eq:udualelliptic}) has a unique solution $\Psi \in W^{1,p^{\prime}}_0\cap W^{2,p^{\prime}}(B,\R^N)$ satisfying
		\begin{equation}
		\norm{\Phi}_{W^{2,r}}\leq C(n,N,r)\norm{T_p\brac{\frac{\Psi}{1+\abs{A}}}}_{L^r}, \quad \mbox{for any }r\in [2,p^{\prime}].
		\end{equation}	 Take $r = \max\{2,\frac{np^{\prime}}{n+p^{\prime}}\}$, which is smaller than $p^{\prime}$, then $\norm{\nabla \Phi}_{L^{p^{\prime}}}$ can be controlled in the following way with the Sobolev embedding 
		\begin{equation}\label{eq:uharmid2}
		\norm{\nabla \Phi}_{L^{p^{\prime}}} \leq C(p,n,N)\norm{\Phi}_{W^{2,r}} \leq C(p,n,N)\norm{T_p\brac{\frac{\Psi}{1+\abs{A}}}}_{L^r}.
		\end{equation}	 When $\abs{y}\leq 1$, it is easy to see that $\abs{T_p(y)}^r \leq \abs{y}^2 \leq \frac{1}{a_1}E_p(y)$. If $\abs{y}>1$, we consider two cases:
		\begin{itemize}[label=\footnotesize{$\bullet$}]
		\item $\frac{2n}{n+2}\leq p <2$, i.e., $\frac{np^{\prime}}{n+p^{\prime}}\leq 2$ and $r=2$: $(p-1)r =2(p-1)\leq p$ as $p<2$;
		\item $1<p<\frac{2n}{n+2}$, i.e., $\frac{np^{\prime}}{n+p^{\prime}}>2$ and $r= \frac{np^{\prime}}{n+p^{\prime}}$: $(p-1)r = \frac{np(p-1)}{n(p-1)+p}<p$.
		\end{itemize} In both cases, we have $\abs{T_p(y)}^r =\abs{y}^{p-1}r \leq \abs{y}^p\leq CE_p(y)$.	 Thus, with (\ref{eq:uharmid1}), (\ref{eq:uharmid2}) and the difference between $u$ and $w$ (see \ref{it:uEkeland1}), the estimate (\ref{eq:uharapprox}) can be obtained as in Proposition \ref{prop:1harapprox}.
	\end{proof}
	
\subsection{Excess decay estimate}\label{subsec:uexcess}
	\par For a ball $B_R=B(x_0,R)\ssubset \Omega$, we define the excess
		\begin{equation}\label{eq:uexcess}
		\EE(x_0,R):= \fint_{B_R}E_p\left(\frac{\nabla u-(\nabla u)_{R}}{1+\lvert (\nabla u)_R\rvert}\right)\id x.
		\end{equation} When the centre $x_0$ is fixed, we will abbreviate the excess as $\EE(R)$.	
	\begin{prop}\label{prop:uexdecay}
	Suppose that $F\colon \R^{N\times n}\to \R$ satisfies \textnormal{\ref{it:intp}, \ref{it:intqc}, \ref{it:int2deriv}} and \textnormal{\ref{it:int2deLip}} with $p\in (1,2)$, and $\omega$ satisfies \textnormal{\ref{it:omega1}}. The map $u\in W^{1,p}_{loc}(\Omega,\R^N)$ is an $\omega$-minimizer of $\F$ with $R_0>0$. For any $\sigma \in(0,1)$, there exists $\varepsilon_1>0$ such that if
		\begin{equation}
		R<R_0, \quad \EE(x_0,R)<\varepsilon_1
		\end{equation} for some ball $B_R = B(x_0,R)\ssubset \Omega$, then we have 
		\begin{equation}\label{eq:uexdecay}
		\EE(\sigma R)\leq c_1\sigma^{-(n+2)}(\EE(R)^s+\omega(R)^p)+c_2 \sigma\EE(R)+c_3\omega(2\sigma R),
		\end{equation} where $s$ is as in Proposition \textnormal{\ref{prop:uharapprox}} and $c_i=c_i(n,N,L,\ell,p)>0$, $i=1,2,3$.
	\end{prop}
	\begin{proof}
	We only consider $\sigma\in (0,\frac{1}{4})$ as it is obvious when $\sigma \in [\frac{1}{4},1)$. As in Proposition \ref{prop:1excessest}, we define $a(x) = u_{B_R} +(\nabla u)_{B_R}(x-x_0)$, $\tilde{u}=u-a$ and $\tilde{F}=F_{\nabla a}$. Let $h$ be the harmonic map determined by (\ref{eq:uellipsystem}) and set \[ \tilde{h}=h-a, \quad a_1(x)=\tilde{h}(x_0)+\nabla \tilde{h}(x_0)(x-x_0),\quad a_0=a+a_1. \] With (\ref{eq:ellipest1}) we have 
		\begin{align}
		&\ \lvert \nabla h(x_0)-(\nabla u)_R\rvert = \lvert \nabla \tilde{h}(x_0)\rvert \leq C\fint_{B_R}\lvert \nabla \tilde{h}\rvert\id x \\
		\leq&\ c_4\fint_{B_R}\lvert \nabla\tilde{u}\rvert\id x \leq c_4 \brac{\fint_{B_R}\abs{\nabla \tilde{u}}^p \id x}^{\frac{1}{p}} \label{eq:uexdecaymid1} 
		\end{align} 
\par In each step, a different normalising factor is needed, and we now give the comparison of them. The first one is as follows:
		\begin{align}
		1+\lvert (\nabla u)_R\rvert & \leq 1 + \lvert (\nabla u)_{\sigma R}\rvert + \sigma^{-n}\fint_{B_R}\lvert \nabla u-(\nabla u)_R\rvert\id x \notag \\
		&\leq  1 + \lvert (\nabla u)_{\sigma R}\rvert + \frac{1+\lvert (\nabla u)_{R}\rvert}{\sigma^n} \brac{\fint_{B_R}\brac{\frac{\lvert \nabla u-(\nabla u)_R\rvert}{1+\lvert (\nabla u)_R\rvert}}^p \id x}^{\frac{1}{p}} \label{eq:uaverest}\\
		&\leq 1 + \lvert (\nabla u)_{\sigma R}\rvert + \frac{1+\lvert (\nabla u)_{R}\rvert}{\sigma^n}(3\EE(R))^{\frac{1}{2p}},\notag
		\end{align} where the last line is from Lemma \ref{lem:Emeanest2} if we take $\varepsilon_1<1$. We further require $\sigma^{-n}(3\varepsilon_1)^{\frac{1}{2p}}<\frac{1}{2}$, i.e., $\varepsilon_1 < \frac{\sigma^{2np}}{3\cdot 4^p}$, then the above estimate gives 
		\begin{equation}\label{eq:uexest1}
		1+\lvert (\nabla u)_R\rvert \leq 2(1+\lvert (\nabla u)_{\sigma R}\rvert).
		\end{equation} For $1+\lvert \nabla h(x_0)\rvert$ and $1+\lvert (\nabla u)_{\sigma R}\rvert$, we have
		\begin{align*}
		\frac{1+\lvert \nabla h(x_0)\rvert}{1+\lvert (\nabla u)_{\sigma R}\rvert} &\leq 1 +\frac{1}{1+\lvert (\nabla u)_{\sigma R}\rvert}(\lvert \nabla h(x_0)-(\nabla u)_R\rvert+\lvert (\nabla u)_R-(\nabla u)_{\sigma R}\rvert)\\
		&\overset{(\ref{eq:uexdecaymid1})}{\leq} 1+ c_4\fint_{B_R}\frac{\lvert \nabla u-(\nabla u)_R\rvert}{1+\lvert (\nabla u)_{\sigma R}\rvert}\id x + \sigma^{-n}\fint_{B_R}\frac{\lvert \nabla u-(\nabla u)_R\rvert}{1+\lvert (\nabla u)_{\sigma R}\rvert}\id x\\
		& \leq 1 +2(c_4+\sigma^{-n})(3\EE(R))^{\frac{1}{2p}},
		\end{align*} where the last line follows from (\ref{eq:uexest1}), H\"{o}lder's inequality and Lemma \ref{lem:Emeanest2}. Taking $2(c_4+\sigma^{-n})(3\varepsilon_1)^{\frac{1}{2p}}<\frac{1}{2}$, i.e., $\varepsilon_1 < \frac{1}{3\cdot 16^p}(c_4+\sigma^{-n})^{-2p}$, we have 
		\begin{equation}\label{eq:uexest2}
		\frac{1+\lvert \nabla h(x_0)\rvert}{1+\lvert (\nabla u)_{\sigma R}\rvert} \leq \frac{3}{2}.  
		\end{equation} The comparison between $1+\abs{(\nabla u)_R}$ and $1+\abs{\nabla h(x_0)}$ is similar:
		\begin{align*}
		\frac{1+\lvert (\nabla u)_R\rvert}{1+\lvert \nabla h(x_0)\rvert} &\leq 1+ \frac{\lvert (\nabla u)_R-\nabla h(x_0)\rvert}{1+\lvert (\nabla u)_R\rvert}\cdot \frac{1+\lvert (\nabla u)_R\rvert}{1+\lvert \nabla h(x_0)\rvert}\\
		&\leq 1+ 2c_4(\EE(R))^{\frac{1}{2p}}\frac{1+\lvert (\nabla u)_R\rvert}{1+\lvert \nabla h(x_0)\rvert}\\
		&\leq 1 +\frac{1}{2}\cdot \frac{1+\lvert (\nabla u)_R\rvert}{1+\lvert \nabla h(x_0)\rvert},
		\end{align*} which implies 
		\begin{equation}\label{eq:uexest3}
		\frac{1+\lvert (\nabla u)_{R}\rvert}{1+\lvert \nabla h(x_0)\rvert}\leq 2.
		\end{equation}
	\par Now we estimate $\EE(\sigma R)$: by (\ref{eq:Eest2}) and (\ref{eq:uexest2}) there holds
		\begin{equation}\label{eq:uexdecaymid3}
		\EE(\sigma R) = \fint_{B_{\sigma R}} E_p\left(\frac{\nabla u-(\nabla u)_{\sigma R}}{1+\lvert (\nabla u)_{\sigma R}\rvert} \right)\id x \leq 16 \fint_{B_{\sigma R}}E_p\left(\frac{\nabla u-\nabla h(x_0)}{1+\lvert \nabla h(x_0)\rvert} \right)\id x.
		\end{equation} The right-hand side can be estimated by the Caccioppoli-type inequality (\ref{eq:ucaccio}) 
		\begin{equation}\label{eq:uexdecaymid4}
		\fint_{B_{\sigma R}}E_p\left(\frac{\nabla u-\nabla h(x_0)}{1+\lvert \nabla h(x_0)\rvert} \right)\id x \leq C\fint_{B_{2\sigma R}}E_p\left(\frac{u-a_0}{2\sigma R(1+\lvert \nabla h(x_0)\rvert)} \right) \id x + C\omega(2\sigma R).
		\end{equation} The term involving $u-a_0$ can be estimated, like in Proposition \ref{prop:1excessest}, by decomposing $u-a_0$ into $\tilde{u}-\tilde{h}$ and $\tilde{h}-a_1$. Applying (\ref{eq:uharapprox}) and (\ref{eq:uexest3}), we have 
		\begin{align*}
		\fint_{B_{2\sigma R}} E_p\left(\frac{\tilde{u}-\tilde{h}}{2\sigma R(1+\abs{\nabla h(x_0)})}\right)\id x &\leq C\sigma^{-(n+2)}\fint_{B_R}E_p\brac{\frac{u-h}{R(1+\abs{(\nabla u)_R})}} \label{eq:uexdecaymid5}\\ 
		&\leq C\sigma^{-(n+2)}\left(\EE(R)^s + \varepsilon_{A,p}^{\frac{p}{2}}+\varepsilon_{A,p}^{\frac{r^{\prime}}{2}}\right).\notag
		\end{align*} The estimate of $\abss{\nabla^2 h(x_0)}$ in Lemma \ref{lem:ellip1} and (\ref{eq:uexest3}) implies
		\begin{equation}\label{eq:uexdecaymid6}
		 \fint_{B_{2\sigma R}} E_p\left(\frac{\tilde{h}-a_1}{2\sigma R(1+\abs{\nabla h(x_0)})}\right)\id x \\
		\leq  E_p\brac{C\sigma \fint_{B_R} \frac{\abs{\nabla u-(\nabla u)_R}}{1+\abs{(\nabla u)_R}}\id x} \leq C\sigma \EE(R),
		\end{equation} where we used (\ref{eq:Eest2}) and Jensen's inequality. Notice that the term $\varepsilon_{A,p}$ can be estimated with the triangle inequality and Lemma \ref{lem:Emeanest2}, and we obtain \[ \varepsilon_{A,p} = \omega(R)\fint_{B_R} \frac{a_6+a_7\lvert \nabla u\rvert^p}{(1+\lvert (\nabla u)_R\rvert)^p}\id x \leq C\omega(R). \]  Thus, combining (\ref{eq:uexdecaymid3})-(\ref{eq:uexdecaymid6}) we have the desired estimate (\ref{eq:uexdecay}) of $\EE(\sigma R)$ under the condition $\varepsilon_1 < \frac{1}{3}\min\{\frac{\sigma^{2np}}{4^p}, \frac{1}{16^p}(c_4+\sigma^{-n})^{-2p}\}$.
	\end{proof} 
	
\subsection{Final conclusion}\label{subsec:ufinal}
	\par In this subsection, we use the excess decay estimate above to further obtain a Morrey's type estimate for $\nabla u$, which then implies the H\"{o}lder regularity of $u$.
	\par For any $\alpha \in(0,1)$, we take $\gamma =p(\alpha-1)+n \in(n-p,n)$.
	\begin{prop}
	Suppose that $F\colon \R^{N\times n}\to \R$ satisfies \textnormal{\ref{it:intp}, \ref{it:intqc}, \ref{it:int2deriv}} and \textnormal{\ref{it:int2deLip}} with $p\in (1,2)$, and $\omega$ satisfies \textnormal{\ref{it:omega1}}. The map $u\in W^{1,p}_{loc}(\Omega,\R^N)$ is an $\omega$-minimizer of $\F$ with $R_0>0$. There exist $R_1 \in (0,R_0), \varepsilon_2 \in (0,1)$ such that for any ball $B_R=B(x_0,R) \ssubset \Omega$ with \[ 0<R<R_1, \quad \EE(x_0,R)<\varepsilon_2, \] we have 
		\begin{equation}
		\int_{B_{\rho}}\abs{\nabla u}^p\id x \leq c_5 \brac{\brac{\frac{\rho}{R}}^{\gamma}\int_{B_R} \abs{\nabla u}^p \id x +\rho^{\gamma}}
		\end{equation} for some $c_5=c_5(n,N,L,\ell,p,\gamma)>0$, where $\gamma \in (n-p,n)$ is defined as above.
	\end{prop} 
	\begin{proof}
	Fix the constants $\sigma, \varepsilon_2$ and $R_1$ in order: 
		\begin{align}
		&\sigma = \min\left\{\frac{1}{2},\frac{1}{4c_2}, 2^{-\frac{2p}{n-\gamma}}\right\},\label{eq:usigma} \\
		&\varepsilon_2 = \min\left\{\varepsilon_1, \left(\frac{\sigma^{n+2}}{4c_1}\right)^{\frac{p}{r^{\prime}-p}}, \frac{1}{3\cdot 16^{p-1}}, \frac{\sigma^{2n}}{3\cdot 4^{p-1}}\right\}, \label{eq:ueps}\\
		& R_1 \in (0,R_0) \mbox{ such that }\omega(R_1)\leq \min\left\{\brac{\frac{\sigma^{n+2}\varepsilon_2}{4c_1}}^\frac{1}{p}, \frac{\varepsilon_2}{4c_3} \right\}, \label{eq:ur}
		\end{align} where $r$ is as in Proposition \ref{prop:uharapprox}, and $\varepsilon_1$ and $c_i, i=1,2,3$, are as in Proposition \ref{prop:uexdecay}.
	\par Suppose that for the ball $B_R=B(x_0,R) \ssubset \Omega$ with some $R\in (0,R_1)$ there holds 
		\begin{equation}
		\EE(x_0,R) <\varepsilon_2.
		\end{equation} We will show that 
		\leqnomode \begin{center} \begin{minipage}{.5\textwidth}
		\begin{align}\label{it:uitk} \tag{\textbf{I}$_k$}
		\EE(\sigma^k R) < \varepsilon_2 
		\end{align} \end{minipage} \end{center}
%		\begin{enumerate}[label=(\textbf{I}$_k$),leftmargin=3\parindent, itemsep=.2em]
%		\item\label{it:uitk} \hspace{5.5cm}$\EE(\sigma^k R) < \varepsilon_2$
%		\end{enumerate} 
holds for any $k\geq 0$ by induction. Obviously, it holds for $k=0$, and we assume that (\ref{it:uitk}) holds for some $k \geq 0$. With our choice of $\varepsilon_2$, Proposition \ref{prop:uexdecay} implies \[ \EE(\sigma^{k+1} R)\leq c_1\sigma^{-(n+2)}(\EE(\sigma^k R)^s+\omega(R)^p)+c_2 \sigma\EE(\sigma^k R)+c_3\omega(2\sigma^{k+1} R), \] where $s=\frac{r^{\prime}}{p}$. By (\ref{eq:usigma})-(\ref{eq:ur}), we know \[ c_1\sigma^{-(n+2)}\varepsilon_2^{s-1} \leq \frac{1}{4},\ c_1\sigma^{-(n+2)}\sqrt{\omega(R)}\leq\frac{\varepsilon_2}{4},\ c_2\sigma \leq\frac{1}{4},\ c_3\omega(R)\leq \frac{\varepsilon_2}{4}, \] which thus gives (\textbf{I}$_{k+1}$). Therefore, we have (\ref{it:uitk}) holds for any $k\in \N$.
	\par With (\ref{it:uitk}) and Lemma \ref{lem:Emeanest2} we have 
		\begin{align*}
		 &\int_{B_{\sigma^{k+1}R}} \lvert \nabla u\rvert^p \leq 2^{p-1}\brac{\int_{B_{\sigma^{k+1}R}} \lvert \nabla u-(\nabla u)_{\sigma^k R}\rvert^p \id x+\omega_n (\sigma^{k+1}R)^n\lvert (\nabla u)_{\sigma^k R}\rvert^p}\\
		&\leq  2^{p-1}(1+\lvert (\nabla u)_{\sigma^k R}\rvert)^p\int_{B_{\sigma^{k+1}R}} \frac{\lvert \nabla u-(\nabla u)_{\sigma^k R}\rvert^p}{(1+\lvert (\nabla u)_{\sigma^k R}\rvert)^p}\id x + 2^{p-1}\sigma^n \int_{B_{\sigma^k R}}\lvert \nabla u\rvert^p\id x\\
		&\leq 2^{p-1}(1+\lvert (\nabla u)_{\sigma^k R}\rvert)^p\omega_n (\sigma^kR)^n \sqrt{3\EE(\sigma^k R)} + 2^{p-1}\sigma^n \int_{B_{\sigma^k R}}\lvert \nabla u\rvert^p \id x\\
		&\leq 2^{p-1}(2^{p-1}\sqrt{3\varepsilon_2}+\sigma^n)\int_{B_{\sigma^k R}}\lvert \nabla u\rvert^p\id x +2^{2(p-1)}\omega_n (\sigma^{k}R)^n \sqrt{3\varepsilon_2}
		\end{align*} From the choice of $\varepsilon_2, \sigma$, it is easy to see \[ 2^{2(p-1)}\sqrt{3\varepsilon_2}\leq 1,\quad 2^{p-1}(2^{p-1}\sqrt{3\varepsilon_2}+\sigma^n) \leq 2^p \sigma^n \leq \sigma^{\frac{\gamma+n}{2}}. \] Set $\lambda(\rho) := \int_{B_{\rho}}\lvert \nabla u\rvert^p \id x$, then the above gives 
		\begin{equation}
		\lambda(\sigma^{k+1}R)\leq \sigma^{\frac{\gamma+n}{2}}\lambda(\sigma^k R)+ \omega_n (\sigma^k R)^n
		\end{equation} for any integer $k\geq 0$. With Lemma 7.3 in \cite{Giusti03} we can further obtain 
		\begin{equation}
		\lambda(t) \leq c_5\left( \left(\frac{t}{R}\right)^{\gamma}\lambda(R)+ t^{\gamma} \right),
		\end{equation} where $c_5=c_5(n,\gamma, \sigma)$. 
	\end{proof} Then by a discussion similar to that at the end of Subsection \ref{subsec:1iteration}, there exists a relatively closed null set $S^{\prime}_u \subset \Omega$ such that $\lvert \nabla u\rvert$ is in the Morrey space $L^{p,\gamma}_{loc}(\Omega\setminus S^{\prime}_u)$. The Sobolev embedding implies that $u$ lies in the Campanato space $\mathcal{L}^{p,\gamma+p}_{loc}(\Omega\setminus S^{\prime}_u,\R^N)$, which is actually $C^{0,\alpha}_{loc}(\Omega\setminus S^{\prime}_u,\R^N)$ as $\gamma=p(\alpha -1)+n$. The proof of Theorem \ref{thm:upartialreg} is then complete.
	\begin{rmk}\label{rmk:uind}
	Theorem \ref{thm:upartialreg} can also be approached by an indirect argument, similar to that in \cite{DGK05} or \cite{FosMin08}, by choosing normalising factors carefully. For such an argument, the Lipschitz continuity of $F^{\pprime}$ can be relaxed to
	\begin{enumerate}[label=(\textbf{H3$_\theFhypoths$})$^{\prime}$,leftmargin=3\parindent, itemsep=.2em]{\usecounter{Fhypoths}} \setcounter{Fhypoths}{2}
		\item\label{it:int2deLipwk} For any $z_1,z_2 \in \R^{N\times n}$ we have \[ \lvert F^{\pprime}(z_1)-F^{\pprime}(z_2)\rvert \leq \nu\left(\frac{\lvert z_1-z_2\rvert}{1+\lvert z_1\rvert+\lvert z_2\rvert}\right)\frac{1}{(1+\lvert z_1\rvert+\lvert z_2\rvert)^{2-p}},  \] where $\nu$ is a concave, non-decreasing function on $[0,\infty)$ with $\nu(0)=\lim_{t\to 0}\nu(t)=0$ and $\nu \leq 1$.
		\end{enumerate}
	\end{rmk}

\subsection*{Acknowledgements} The author would like to thank Jan Kristensen for his guidance throughout the project and Franz Gmeineder for helpful discussions and suggestions. Part of the work was done during a visit at the School of Mathematical Sciences at Peking University, and the author would like to express gratitude for the support from the institution.

% Bibliography style
\bibliographystyle{alpha}
\nocite{*}
\bibliography{partialregsubq}

\end{document}